\documentclass[11pt, reqno]{amsart}

\usepackage{iftex}

\ifPDFTeX
  \usepackage[T1]{fontenc}
  \usepackage{lmodern}
  \usepackage{xcolor}
  \usepackage{hyperref}
\else
  \usepackage[T1]{fontenc}
  \usepackage{lmodern}
  \usepackage[dvipdfmx]{xcolor}
  \usepackage[dvipdfmx]{hyperref}
\fi

\definecolor{darkblue}{rgb}{0.0, 0.0, 0.55}
\definecolor{darkmagenta}{rgb}{0.55, 0.0, 0.55}
\definecolor{darkcyan}{rgb}{0.0, 0.55, 0.55}

\hypersetup{
  colorlinks = true,
  linkcolor  = darkblue,
  citecolor  = darkmagenta,
  urlcolor   = darkcyan
}

\usepackage{amssymb, amsthm}
\usepackage{mathtools}
\usepackage{mathrsfs}

\usepackage[all]{xy}

\usepackage[
  left=2.5cm,
  right=2.5cm,
  top=2.7cm,
  bottom=2.3cm,
]{geometry}

\usepackage{enumitem}
\setlist[enumerate]{
  label=\textup{(\arabic*)},
  font=\normalfont,
  leftmargin=25pt
}
\setlist[itemize]{
  font=\normalfont,
  leftmargin=25pt
}

\newcommand{\al}{\alpha}
\newcommand{\be}{\beta}

\newcommand{\la}{\lambda}

\newcommand{\si}{\sigma}

\newcommand{\vph}{\varphi}

\newcommand{\om}{\omega}

\newcommand{\La}{\Lambda}

\newcommand{\PP}{{\mathbb P}}

\newcommand{\ZZ}{{\mathbb Z}}
\newcommand{\NN}{{\mathbb N}}

\newcommand{\fm}{{\mathfrak m}}

\def\ang#1{{\langle #1 \rangle}}
\def\rnum#1{\expandafter{\romannumeral #1}}
\def\Rnum#1{\uppercase\expandafter{\romannumeral #1}}

\renewcommand{\mod}{\operatorname{\mathsf{mod}}}
\newcommand{\Mod}{\operatorname{\mathsf{Mod}}}
\newcommand{\proj}{\operatorname{\mathsf{proj}}}
\newcommand{\grmod}{\operatorname{\mathsf{grmod}}}
\newcommand{\grproj}{\operatorname{\mathsf{grproj}}}
\newcommand{\GrMod}{\operatorname{\mathsf{GrMod}}}
\newcommand{\CM}{\operatorname{\mathsf{CM}^{\mathbb Z}}}

\newcommand{\uCM}{\operatorname{\underline{\mathsf{CM}}^{\mathbb Z}}}
\newcommand{\uCMz}{\operatorname{\underline{\mathsf{CM}}_{0}^{\mathbb Z}}}
\newcommand{\qgr}{\operatorname{\mathsf{qgr}}}
\newcommand{\fdim}{\operatorname{\mathsf{fdim}}}
\newcommand{\coh}{\operatorname{\mathsf{coh}}}
\newcommand{\D}{\operatorname{\mathsf{D}}}
\newcommand{\K}{\operatorname{\mathsf{K}}}
\newcommand{\thick}{\operatorname{\mathsf{thick}}}

\newcommand{\Ind}{\operatorname{Ind}}
\newcommand{\Hom}{\operatorname{Hom}}
\newcommand{\End}{\operatorname{End}}

\newcommand{\Ext}{\operatorname{Ext}}

\newcommand{\GrAut}{\operatorname{GrAut}}
\newcommand{\RGamma}{\operatorname R\!\Gamma}

\newcommand{\gr}{\operatorname{gr}}
\newcommand{\RHom}{\operatorname R\!\Hom}
\newcommand{\Lotimes}{\stackrel{\operatorname L}{\otimes}}

\newcommand{\gldim}{\operatorname{gldim}}
\newcommand{\injdim}{\operatorname{injdim}}
\newcommand{\projdim}{\operatorname{projdim}}

\newcommand{\rad}{\operatorname{rad}}
\renewcommand{\H}{\operatorname{H}}
\newcommand{\op}{\operatorname{op}}
\newcommand{\en}{\operatorname{e}}
\newcommand{\id}{\operatorname{id}}
\newcommand{\CMreg}{\operatorname{CMreg}}
\newcommand{\Extreg}{\operatorname{Extreg}}

\newcommand{\Proj}{\operatorname{Proj}}

\theoremstyle{plain} 
\newtheorem{thm}{Theorem}[section]
\newtheorem{cor}[thm]{Corollary}
\newtheorem{lem}[thm]{Lemma}
\newtheorem{prop}[thm]{Proposition}

\theoremstyle{definition}
\newtheorem{dfn}[thm]{Definition}
\newtheorem{ex}[thm]{Example}

\newtheorem{rem}[thm]{Remark}

\numberwithin{equation}{section}

\begin{document}

\title
[Trivial extensions of Koszul Artin-Schelter regular algebras]
{Trivial extensions of \\ Koszul Artin-Schelter regular algebras}

\author{Kenta Ueyama}
\address{Department of Mathematics, Faculty of Science, Shinshu University, 3-1-1 Asahi, Matsumoto, Nagano 390-8621, Japan}
\email{ueyama@shinshu-u.ac.jp}



\begin{abstract}
Let $S$ be an $\mathbb N$-graded Koszul Artin-Schelter regular algebra and let
$\sigma$ be a graded algebra automorphism of $S$.
We study the stable category of graded maximal Cohen-Macaulay modules over the trivial extension algebra $S\ltimes S_\sigma(-1)$. We show that this category is triangle equivalent to the bounded derived category of finitely generated (ungraded) modules over the Koszul dual algebra of the Zhang twist $S^{\sigma^{-1}}$.
In the connected graded case, we also obtain a criterion for when two such stable categories are triangle equivalent, and show that such an equivalence induces an equivalence between the categories of graded modules over the original algebras.
\end{abstract}

\maketitle

\section{Introduction}

Trivial extension algebras play an essential role in both commutative algebra and the representation theory of finite-dimensional algebras, since they form a rich source of important examples in these areas
(see e.g.\ \cite{AW, FGR, SY}).
In particular, graded trivial extensions arise naturally in the representation theory of finite-dimensional Gorenstein algebras
(see e.g.\ \cite{Ch, Hapb, MY}).

In recent years, the stable categories of graded maximal Cohen-Macaulay modules over graded Gorenstein algebras have been widely studied as triangulated categories related to various areas including representation theory, commutative algebra, and noncommutative algebraic geometry (see e.g.\ \cite{AIR, BIY, Han, HIMO, IKU, IT, MUs, MUk, SV}).

Motivated by these developments, in this paper we study graded Gorenstein algebras arising from trivial extensions that are not necessarily commutative or finite-dimensional.
More precisely, we investigate the stable category $\uCM(S \ltimes L)$ of graded right maximal Cohen-Macaulay modules over the trivial extension algebra $S \ltimes L$, where $S$ is a $\NN$-graded Koszul Artin-Schelter (AS) regular algebra and $L$ is a graded invertible $S$-bimodule generated in degree $1$.
Our main result describes this category in terms of the Koszul dual of a twisted algebra of $S$.

\begin{thm}[Theorem \ref{tmain1a}]\label{tmain1}
Let $S$ be an $\NN$-graded Koszul AS-regular algebra of dimension $d$ and let $\sigma\in\GrAut(S)$ be a graded algebra automorphism of $S$. Consider the graded $S$-bimodule
$L := S_\sigma(-1)$ in $\GrMod S^{\en}$.
Then there are equivalences of triangulated categories
\begin{equation*}
\xymatrix@R=0.5cm@C=1.25cm{
\uCM(S \ltimes L)\ar[r]^-{\sim} 
& \D^b(\mod ((S^{\sigma^{-1}})^!)^{\op})\\
\thick(\Omega^d (S \ltimes L)_0(d)) \ar[r]^-{\sim} \ar@{}[u]|{\rotatebox{90}{$\subset$}}
& \K^b(\proj ((S^{\sigma^{-1}})^!)^{\op}), \ar@<2ex>@{}[u]|{\rotatebox{90}{$\subset$}}
}
\end{equation*}
where $(S^{\sigma^{-1}})^!$ denotes the Koszul dual algebra of the Zhang twist $S^{\sigma^{-1}}$, 
and $\thick(\Omega^d (S \ltimes L)_0(d))$ is the thick subcategory of $\uCM(S \ltimes L)$ generated by $\Omega^d (S \ltimes L)_0(d)$.
\end{thm}

Theorem \ref{tmain1} can be viewed as a form of Koszul duality. However, unlike the usual form of Koszul duality, which relates derived categories of graded modules, the equivalences in Theorem \ref{tmain1} involve the derived category of ungraded modules.

As a special case, when $S$ is connected graded and $L = \omega_S(d-1)$ is the shifted canonical bimodule, Theorem~\ref{tmain1} gives an equivalence
\begin{equation*}
\uCM(S \ltimes \omega_S(d-1)) \simeq \D^b(\mod ((S^{\nu^{-1}})^!)^{\op}),
\end{equation*}
where $\nu$ is the Nakayama automorphism of $S$ (see Corollary \ref{cmain1}).
In particular, when $S = k[x_1,\dots,x_d]$, this equivalence specializes to
\begin{equation*}
\uCM(k[x_1,\dots,x_d,y]/(y^2)) 
\simeq 
\D^b(\mod \Lambda^d),
\end{equation*}
where $\Lambda^d$ denotes the exterior algebra in $d$ variables.

The key ingredient in the proof of Theorem \ref{tmain1} is the following BGG-type correspondence.

\begin{thm}[Theorem \ref{t.qcd}]\label{t.qcdintro}
Let $A$ be an $\NN$-graded Koszul AS-Gorenstein algebra.
Assume that $A^!$ is AS-Gorenstein.
Then there exists a duality $\uCM(A) \to \D^b(\qgr A^!)$,
where $\qgr A^!$ is the Serre quotient of the category of finitely generated graded right $A^!$-modules modulo the subcategory of finite-dimensional modules.
\end{thm}

Next, in the connected graded case, we study when two stable categories as in Theorem \ref{tmain1} are triangle equivalent.

\begin{thm}[Theorem \ref{tmain2a}]\label{tmain2}
Let $S$ and $R$ be connected graded Koszul AS-regular algebras,
and let $\sigma\in\GrAut(S)$ and $\tau\in\GrAut(R)$.
Consider the graded bimodules
$L := S_\sigma(-1)$ in $\GrMod S^{\en}$ and $N := R_\tau(-1)$ in $\GrMod R^{\en}$. Then the following conditions are equivalent:
\begin{enumerate}[label=(\roman*)]
\item $\uCM(S \ltimes L) \simeq \uCM(R \ltimes N)$ as triangulated categories.
\item $S^{\sigma^{-1}} \cong R^{\tau^{-1}}$ as graded algebras.
\end{enumerate}
In particular, if these conditions hold, then we have an equivalence $\GrMod S \simeq \GrMod R$.
\end{thm}

It is noteworthy that a triangle equivalence between the stable categories associated to the trivial extensions $S \ltimes L$ and $R \ltimes N$ forces an equivalence between the categories of graded modules over the original algebras $S$ and $R$.

In particular, when $L = \omega_S(d-1)$ and $N = \omega_R(d'-1)$, Theorem \ref{tmain2} implies
\[
S \cong R \ \ \Longrightarrow \ \  \uCM(S \ltimes \omega_S(d-1)) \simeq \uCM(R \ltimes \omega_R(d'-1)) \ \ \Longrightarrow \ \ \GrMod S \simeq \GrMod R
\]
(see Corollary \ref{cmain2}).
By considering the case where $S$ and $R$ are skew polynomial algebras in three variables, we find that the converse of each implication above does not hold in general (see Example~\ref{exntd3}).

This paper is organized as follows. Section 2 collects notation, definitions, and basic results needed for later sections, including AS-regular algebras, Koszul algebras, and maximal Cohen-Macaulay modules.
In Section 3, we study Koszul duality and prove Theorem \ref{t.qcdintro}.
In Section 4, we discuss trivial extensions of graded algebras and their relations to Ore extensions and Zhang twists.
Section 5 is devoted to the proofs of Theorems \ref{tmain1} and \ref{tmain2}.
Finally, Section 6 presents examples of the main results.

\section{Preliminaries}

In this section, we fix notation and recall definitions and basic results that will be needed in the sequel.

Throughout this paper, $k$ denotes a field, and all vector spaces, algebras, and tensor products are taken over $k$ unless otherwise specified.

Let $A$ be an algebra.
We write $\Mod A$ (resp.\ $\mod A$) for the category of right $A$-modules (resp.\ finitely generated right $A$-modules).
We denote by $\proj A$ the full subcategory of $\mod A$ consisting of finitely generated projective modules.
Let $A^{\op}$ be the opposite algebra of $A$, and define
$A^{\en} := A^{\op} \otimes A$ to be the enveloping algebra of $A$.
Then the category of left $A$-modules identifies with $\Mod A^{\op}$, and the category of $A$-bimodules identifies with $\Mod A^{\en}$.

Let $\mathcal{A}$ be an abelian category. We denote by $\K(\mathcal{A})$ the homotopy category of complexes over $\mathcal{A}$, and by $\D(\mathcal{A})$ its derived category. We write $\K^-(\mathcal{A})$, $\K^+(\mathcal{A})$, $\K^b(\mathcal{A})$ (resp.\ $\D^-(\mathcal{A})$, $\D^+(\mathcal{A})$, $\D^b(\mathcal{A})$) for the full subcategories consisting of complexes that are  bounded above, bounded below, and bounded, respectively.

Let $\mathcal{T}$ be a triangulated category. 
For a collection $\mathcal{X}$ of objects in $\mathcal{T}$, 
we denote by $\thick(\mathcal{X}) \subset \mathcal{T}$ the smallest thick subcategory containing $\mathcal{X}$.

\subsection{Graded modules}
Let $A=\bigoplus_{i \in \ZZ} A_i$ be a $\ZZ$-graded algebra.
We write $A_{\geq s} := \bigoplus_{i \geq s} A_i$.
We say that $A$ is \emph{$\NN$-graded} if $A_i=0$ for all $i<0$, and that it is \emph{connected graded} if $A$ is $\NN$-graded and $A_0 = k$. 

We write $\GrMod A$ (resp.\ $\grmod A$) for the category of $\ZZ$-graded right $A$-modules (resp.\ finitely generated $\ZZ$-graded right $A$-modules) with degree-preserving $A$-module homomorphisms. 
We denote by $\grproj A$ the full subcategory of $\grmod A$ consisting of finitely generated graded projective modules.
For a graded module $M \in \GrMod A$ and an integer $n$,
we define the \emph{shift} $M(n) \in \GrMod A$ by $M(n)_i := M_{n+i}$ for all $i \in \ZZ$.
For $M, N \in \GrMod A$, we set
$\Ext^i_A(M,N)
:= \bigoplus_{s\in \ZZ} \Ext^i_{\GrMod A}(M, N(s))$.

Let $\GrAut(A)$ denote the group of graded algebra automorphisms of $A$.
We define the \emph{twist} of a graded bimodule by a graded algebra automorphism as follows.

\begin{dfn}
Let $A$ be a $\ZZ$-graded algebra and let $\sigma \in \GrAut(A)$.
For a graded $A$-bimodule $M$,
we define the graded $A$-bimodule $M_\sigma$ by setting
$M_\sigma := M$ with bimodule action $b \cdot m \cdot a := b\, m\, \sigma(a)$ for $a,b\in A, m\in M$.
\end{dfn}

A graded bimodule $L \in \GrMod A^{\en}$ is \emph{graded invertible} if there exists a graded bimodule $N \in \GrMod A^{\en}$ such that $L \otimes_A N \cong N \otimes_A L\cong A$ in $\GrMod A^{\en}$. If $L \cong A_\sigma(s)$ for some $\sigma \in \GrAut(A)$ and $s \in \ZZ$, then $L$ is a graded invertible bimodule.
By \cite[Proposition 1.10]{Ye}, the converse holds if $A$ is connected graded.

\subsection{Noncommutative projective schemes}
Let $A$ be a (right and left) noetherian $\NN$-graded algebra.
We denote by $\fdim A$ the full subcategory of $\grmod A$ consisting of finite-dimensional modules over $k$.
The Serre quotient category
\[ \qgr A := \grmod A/\fdim A\]
is an abelian category.
If $A$ is a connected graded commutative algebra finitely generated in degree $1$, then by Serre's theorem 
there is an equivalence of categories
$\qgr A \simeq \coh(\Proj A)$,
where $\coh(\Proj A)$ denotes the category of coherent sheaves on the projective scheme $\Proj A$.
Motivated by this correspondence, the category $\qgr A$ is called the \emph{noncommutative projective scheme} associated with $A$.
The study of noncommutative projective schemes is one of the central topics in noncommutative algebraic geometry; see \cite{AZ} for foundational results.

\subsection{AS-regular algebras and Koszul algebras}

Artin-Schelter (AS) regular algebras, defined below, provide a natural noncommutative analogue of polynomial rings and play a central role in noncommutative algebraic geometry.

\begin{dfn}
An $\NN$-graded algebra $A$ is called
\emph{AS-regular} (resp.\ \emph{AS-Gorenstein}) if
\begin{itemize}
\item $A$ is noetherian. 
\item $A_0$ is a basic finite-dimensional semisimple algebra, i.e., $A_0$ is isomorphic to a finite direct product of copies of $k$,
\item $\gldim A = d < \infty$ (resp.\ $\injdim_A A = \injdim_{A^{\op}} A = d < \infty$),
\item $\RHom_A(-,A)[d]$ induces a bijection from the set of isomorphism classes of graded simple right $A$-modules to the set of isomorphism classes of graded simple left $A$-modules.
\end{itemize}
We call $d$ the \emph{dimension} of $A$.

In particular, if $A$ is connected graded,
then the fourth condition is equivalent to the following:
\begin{itemize}
\item $\Ext^i_A(k, A) \cong
\begin{cases}
k(\ell)\ \text{for some $\ell\in\ZZ$}& \text{if } i = d,\\
0       & \text{if } i \neq d.
\end{cases}$
\end{itemize}
We call $\ell$ the \emph{Gorenstein parameter} of $A$.
\end{dfn}

\begin{rem}
\begin{enumerate}
\item Let $A$ be an AS-regular algebra. For any simple graded module $N \in \GrMod A$, we have $\dim_k\Ext_A^i(N,A)=\delta_{i,d}$ and $\projdim_AN=d$.
\item For variations of the fourth condition in the definition of AS-regular algebras, see \cite[Section 5.1]{RR}.
\end{enumerate}

\end{rem}

The following result is a consequence of
\cite[Propositions 3.4, 3.5 and Theorem 3.6]{Le}.
\begin{prop}\label{pLe}
Let $A$ be an $\NN$-graded algebra and let $f$ be a homogeneous regular normal element of $A$ of positive degree.
Then $A$ is AS-Gorenstein of dimension $d+1$ if and only if $A/(f)$ is AS-Gorenstein of dimension $d$.
\end{prop}

\begin{dfn}
An $\NN$-graded algebra $A$ is called \emph{Koszul} if 
\begin{itemize}
\item $A$ is locally finite, i.e., $\dim_kA_i<\infty$ for all $i$.
\item $A_0$ is a basic finite-dimensional semisimple algebra,
\item $A_0 \in \GrMod A$ has a linear projective resolution (i.e., a graded projective resolution whose $i$-th term is generated in degree $i$).
\end{itemize}
\end{dfn}

We briefly recall some basic facts on Koszul algebras from \cite{BGS}.
Let $A$ be a Koszul algebra. Then $A$ is a quadratic algebra, that is, it can be expressed as $A = T_{A_0}(V)/(R)$, where $T_{A_0}(V)$ is the tensor algebra of the $A_0$-bimodule $V:=A_1$, and $R$ is an $A_0$-subbimodule of $V \otimes_{A_0} V$.
The \emph{quadratic dual} (or \emph{Koszul dual}) of $A$ is defined by $A^! := T_{A_0}(V^*)/(R^\perp)$, where $(-)^* = \Hom_{A_0}(-,A_0)$, and $R^\perp \subset V^* \otimes_{A_0} V^*$ is the annihilator of $R$ under the natural identification $(V \otimes_{A_0} V)^* \cong V^* \otimes_{A_0} V^*$.
It is known that $A^!$ is again Koszul, and that there is an isomorphism of graded algebras $A^! \cong \bigoplus_{i \in \mathbb{N}} \Ext_A^i(A_0, A_0)^{\mathrm{op}}$.
Moreover, a graded isomorphism of Koszul algebras $A \xrightarrow{\sim} B$ induces a corresponding graded isomorphism of their duals $B^! \xrightarrow{\sim} A^!$.

By \cite[Theorem 5.11]{Sm}, if $A$ is a connected graded AS-regular algebra of dimension $d$ with Hilbert series $(1-t)^{-d}$, then $A$ is Koszul.

The following notions are fundamental in the study of Koszul algebras and AS-regular algebras.

\begin{dfn}
Let $A$ be a $\ZZ$-graded algebra and let
$\si \in \GrAut(A)$.
\begin{enumerate}
\item The \emph{graded Ore extension} $A[x;\sigma]$ is the $\ZZ$-graded algebra defined by $A[x;\sigma] := \bigoplus_{i\ge 0} A x^i$,
where $\deg x = s$ for some $s \in \ZZ$, so that $\deg(a x^i)=\deg a + i s$ for homogeneous $a\in A$, and whose multiplication is determined by the relation
\[
x a = \sigma(a) x
\]
for all $a\in A$.
\item The \emph{Zhang twist} of $A$ by $\si$, denoted by $A^\si$, is the
$\ZZ$-graded algebra whose underlying graded vector space is $A$, and
whose multiplication $\ast$ is defined by
\[
a \ast b = a \sigma^{i}(b)
\]
for all homogeneous elements $a \in A_i$ and $b \in A$.
\end{enumerate}
\end{dfn}

The Zhang twist is an important operation on graded algebras, since it yields equivalences between categories of graded modules.
Let $A$ be a $\ZZ$-graded algebra and let $\sigma \in \GrAut(A)$.
For $M \in \GrMod A$, we define $M^\sigma \in \GrMod A^\sigma$ by setting $M^\sigma = M$ as a graded vector space, with the $A^\sigma$-module structure given by $m \cdot a := m\,\sigma^i(a)$
for all homogeneous elements $m \in M_i$ and $a \in A$.
For $\varphi \in \Hom_{\GrMod A}(M,N)$, we define
$\varphi^\sigma \in \Hom_{\GrMod A^\sigma}(M^\sigma,N^\sigma)$ by
$\varphi^\sigma(m)=\varphi(m)$.
The following is Zhang's theorem.

\begin{thm}\cite[Theorems 3.1, 3.4]{Zh} \label{tZ}
Let $A$ be a $\ZZ$-graded algebra and let $\sigma \in \GrAut(A)$. Then $F:=(-)^\si : \GrMod A \to \GrMod A^\si$ is an equivalence functor such that $F(eA(i)) \cong eA^\si(i)$ for all $i \in \ZZ$ and all idempotents $e \in A$.
\end{thm}

We will use the following results later.

\begin{prop}\label{poz}
\begin{enumerate}
\item
If $A$ is an AS-regular (resp.\ AS-Gorenstein) algebra of dimension $d$ and $\sigma \in \GrAut(A)$, then $A[x;\sigma]$, equipped with a grading such that $\deg x \ge 1$, is AS-regular (resp.\ AS-Gorenstein) of dimension $d+1$.
\item  
If $A$ is a Koszul algebra and $\sigma \in \GrAut(A)$, then $A[x;\sigma]$, equipped with the grading $\deg x = 1$, is Koszul.
\item 
If $A$ is an AS-regular (resp.\ AS-Gorenstein) algebra of dimension $d$ and $\sigma \in \GrAut(A)$, then $A^\sigma$ is AS-regular (resp.\ AS-Gorenstein) of dimension $d$.
\item 
If $A$ is a Koszul algebra and $\sigma \in \GrAut(A)$, then $A^\sigma$ is Koszul.
\end{enumerate}
\end{prop}

\begin{proof}
(1) It is well-known that $\gldim A[x;\sigma] = \gldim A + 1$; see \cite[Theorem 5.3(iii)]{MR}. 
Since $x \in A[x;\sigma]$ is a regular normal element and $A[x;\sigma]/(x) \cong A$,  the result follows from Proposition~\ref{pLe}.

(2) Since $x \in A[x;\sigma]$ is a regular normal element of degree $1$ and $A[x;\sigma]/(x) \cong A$, the result follows from the same argument as in the proof of \cite[Theorem 1.5]{ST}.

(3) By \cite[Proposition 5.1]{Zh}, $A^{\sigma}$ is noetherian.
Since $\dim_k \Ext^i_{A^{\sigma}}(M^{\sigma},A^{\sigma})=\dim_k \Ext^i_A(M, A)$, 
we have $\injdim_{A^{\sigma}} A^{\sigma} = \injdim_A A$.
The left version follows similarly by \cite[Corollary 4.4]{Zh}.

Let $N'$ be a graded simple right $A^\sigma$-module. Then there exists a graded simple right $A$-module $N$ such that the graded minimal projective resolutions of $N'$ (over $A^\sigma$) and $N$ (over $A$) have the same graded Betti numbers by Theorem \ref{tZ}. Thus we have $\projdim_{A^\sigma} N' = \projdim_A N$, and hence $\gldim A^\sigma = \gldim A$.
Moreover, we have 
$\dim_k \Ext^i_{A^\sigma}(N', A^\sigma) = \dim_k \Ext^i_A(N, A) = \delta_{i,d}$,
so $\Ext^d_{A^\sigma}(N', A^\sigma)$ is a graded simple left $A^\sigma$-module.  This proves the assertion.

(4) By Theorem \ref{tZ}, the graded minimal projective resolutions of $(A^\sigma)_0 \cong (A_0)^\sigma$ (over $A^\sigma$) and $A_0$ (over $A$) have the same graded Betti numbers. Hence the result follows.
\end{proof}

\subsection{Canonical modules and maximal Cohen-Macaulay modules}

Let $A$ be an $\NN$-graded algebra and put $\fm=A_{\geq 1}$. We define the functor
$\Gamma_{\fm} : \GrMod A \to \GrMod A$ by
$\Gamma_{\fm}(-) = \varinjlim_n \Hom_{A}(A/A_{\ge n},-)$.
The derived functor of $\Gamma_{\fm}$ is denoted by $\RGamma_{\fm}$ and its cohomologies are denoted by
\[ \H_{\fm}^i(-):= h^i \RGamma_{\fm}(-)=\varinjlim_n\Ext^i_{A}(A/A_{\ge n},-). \]
Similarly, we define $\Gamma_{\fm^{\op}}$, $\RGamma_{\fm^{\op}}$, and $\H_{\fm^{\op}}^i(-)$ for $A^{\op}$.

Let $A$ be an AS-Gorenstein algebra of dimension $d$.
Then $\H_{\fm}^i(A) = \H_{\fm^{\op}}^i(A) = 0$ for all $i \neq d$. Moreover, $\H_{\fm}^d(A) \cong \H_{\fm^{\op}}^d(A)$ in $\GrMod A^{\en}$ by \cite[Lemma 2.12]{IKU}.

\begin{dfn}\label{define omega}
Let $A$ be an AS-Gorenstein algebra of dimension $d$.
We call
\[
\omega_A := D\H_{\fm}^d(A) \cong D\H_{\fm^{\op}}^d(A)  \in \GrMod A^{\en} 
\]
the \emph{canonical module} of $A$, where $D=\Hom_k(-,k)$ is the graded $k$-dual.
\end{dfn}

The following proposition summarizes important properties of canonical modules.

\begin{prop} \label{p.canm}
Let $A$ be an AS-Gorenstein algebra of dimension $d$.
\begin{enumerate}
\item \cite[Proposition 2.17(2)]{IKU}
$\omega_A$ is a graded invertible $A$-bimodule.
\item \cite[Proposition 2.17(3)]{IKU}
We have a duality
\[\RHom_A(-,\omega_A):\D^b(\grmod A)\rightleftarrows\D^b(\grmod A^{\op}):\RHom_{A^{\op}}(-,\omega_A).\] 
\item \cite[Theorem 2.16]{IKU} \textnormal{(Local duality)} For $X \in \D^{-}(\GrMod A)$, we have 
\[ \RHom_A(X, \omega_A)[d]
\cong \RHom_A(X, D\RGamma_{\fm}(A)) \cong 
D\RGamma_{\fm}(X)
\quad \text{in}\ \D(\GrMod A^{\op}).
\]
\end{enumerate}
\end{prop}

If $A$ is a connected graded AS-Gorenstein algebra of dimension $d$ and Gorenstein parameter $\ell$,
then there is a graded algebra automorphism
$\nu \in \GrAut(A)$ such that 
\[
\omega_A \cong A_{\nu}(-\ell)\ \ \text{in}\ \GrMod A^{\en}
\]
by \cite[Theorem 1.2]{JoL}. This graded algebra automorphism $\nu$ is called the \emph{Nakayama automorphism} of $A$.
Since $A$ is connected graded, we see that the Nakayama automorphism of $A$ is uniquely determined.

Let $A$ be an AS-Gorenstein algebra of dimension $d$. A finitely generated module $M \in \grmod A$ is called \emph{maximal Cohen-Macaulay} if $\Ext^i_A(M, A)=0$ for all $i \neq 0$. 
Note that $M \in \grmod A$ is maximal Cohen-Macaulay if and only if $\Ext_A^i(M, \omega_A) = 0$ for all $i \neq 0$, if and only if $\H_{\fm}^i(M) = 0$ for all $i \neq d$.
Let $\CM(A)$ denote the full subcategory of $\grmod A$ consisting of graded maximal Cohen-Macaulay modules. 
Then $\CM(A)$ is a Frobenius category.
The \emph{stable category} of graded maximal Cohen-Macaulay modules, denoted by $\uCM(A)$, has the same objects as $\CM(A)$,
and the morphism space is given by
$\Hom_{\uCM(A)}(M, N) = \Hom_{\CM(A)}(M,N)/P(M,N)$, 
where $P(M,N)$ is the subspace of degree-preserving $A$-module homomorphisms that factor through a graded projective module. Since $\CM(A)$ is a Frobenius category, $\uCM(A)$ admits a canonical triangulated structure; see \cite{Hapb}.

\section{BGG-type correspondence for Koszul AS-Gorenstein algebras}

The Bernstein-Gelfand-Gelfand (BGG) correspondence is the equivalence 
\[
\underline{\grmod}\,\Lambda^d \simeq \D^b(\coh \PP^{d-1}),
\]
where $\Lambda^d$ is the graded exterior algebra on $d$ generators in degree $1$. 
Various generalizations of this equivalence have been obtained using Koszul duality (e.g.\ \cite{BGS, BEH, JoB, MS, MoR}).
Here, we provide a further result in this direction.

Let $A$ be a Koszul algebra, and set $E(A) := \bigoplus_{i \in \NN} \Ext_A^i(A_0, A_0)$. 
Then $A^! \cong E(A)^{\op} \cong E(A^{\op})$. 
For simplicity, we write $E := E(A)$ and identify $A^!$ with $E^{\op}$ via this isomorphism. Since 
\begin{align*} 
A_1 \otimes_{A_0} E_1	\cong 
A_1\otimes _{A_0}\Hom_{A_0}(A_1, A_0)\cong \Hom_{A_0}(A_1, A_1\otimes _{A_0}A_0)\cong \Hom_{A_0}(A_1, A_1), 
\end{align*}
we can choose elements
$v_{\la}\in A_1$ and ${v^*_{\la}}\in E_1$
such that $\sum_\la v_{\la}\otimes {v_{\la}^*}$ corresponds to $\id _{A_1}$
under the above isomorphisms.

We define full subcategories of $\K(\GrMod A)$ as follows:
\begin{itemize}
\item $\K^{\uparrow}(\GrMod A):=\{X\in \K^{-}(\GrMod A)\mid X^i_j=0\ \text{for}\ i\gg 0\ \text{or}\ i+j\ll 0 \}$,
\item $\K^{\downarrow}(\GrMod A):=\{X\in \K^{+}(\GrMod A)\mid X^i_j=0\ \text{for}\ i\ll 0\ \text{or}\ i+j\gg 0 \}$.
\end{itemize}

Define the functor $F=F_A: \K(\GrMod A) \to \K(\GrMod E)$ by $F(X)^\ell_m=\bigoplus_{j}X^{\ell-j}_{j}\otimes_{A_0}E_{m+j}$ with differentials $d_F=d_{F}'+d_{F}''$ given by
\begin{align*}
&d_{F}':X^{\ell-j}_{j}\otimes_{A_0}E_{m+j} \to X^{\ell-j}_{j+1}\otimes_{A_0}E_{m+j+1};\;\; d_{F}'(x\otimes a)= (-1)^{\ell}\sum_\la xv_\la \otimes {v^*_\la}a, \\
&d_{F}'':X^{\ell-j}_{j}\otimes_{A_0}E_{m+j} \to X^{\ell-j+1}_{j}\otimes_{A_0}E_{m+j};\;\; d_{F}''(x\otimes a)= \partial(x)\otimes a,
\end{align*}
where $\partial$ is the differential of $X$.

Define the functor $G=G_A: \K(\GrMod A) \to \K(\GrMod E)$ by $G(X)^\ell_m=\bigoplus_{j}\Hom_{A_0}(E_{-m-j},X^{\ell-j}_{j})$ with differentials $d_G=d_{G}'+d_{G}''$ given by
\begin{align*}
&d_{G}':\Hom_{A_0}(E_{-m-j},X^{\ell-j}_{j})\to \Hom_{A_0}(E_{-m-j-1},X^{\ell-j}_{j+1});\;\; (d_{G}'(f))(a)= (-1)^{\ell-j}\sum_\la  f(a{v^*_\la})v_\la, \\
&d_{G}'':\Hom_{A_0}(E_{-m-j},X^{\ell-j}_{j})\to \Hom_{A_0}(E_{-m-j},X^{\ell-j+1}_{j});\;\; (d_{G}''(f))(a)= \partial f(a).
\end{align*}
In particular, applying this construction to $E$ yields a functor $G_E:\K(\GrMod E) \to \K(\GrMod A)$.

Let $\D^{\uparrow}(\GrMod A)$ (resp.\ $\D^{\downarrow}(\GrMod A)$) be the derived category obtained by inverting quasi-isomorphisms in
$\K^{\uparrow}(\GrMod A)$ (resp.\ $\K^{\downarrow}(\GrMod A)$).
Let $\D^{\uparrow}_{lf}(\GrMod A)$ (resp.\ $\D^{\downarrow}_{lf}(\GrMod A)$)  be the full subcategory of $\D^{\uparrow}(\GrMod A)$ (resp.\ $\D^{\downarrow}(\GrMod A)$) consisting of complexes with locally finite cohomology. 
The following equivalences were discovered by Beilinson-Ginzburg-Soergel \cite{BGS}. 

\begin{thm}[{\cite [Theorems 2.12.1, 2.12.5]{BGS}}] \label{t.BGS}
If $A$ is a Koszul algebra, then $F_A, G_E$ induce mutually quasi-inverse equivalences
\[\xymatrix@C=1pc@R=1pc{
F_A:\D^{\downarrow}(\GrMod A) \ar@<0.5ex>[r]
&\D^{\uparrow}(\GrMod E):G_{E} \ar@<0.5ex>[l]\\
F_A:\D^{\downarrow}_{lf}(\GrMod A)   \ar@<0.5ex>[r] \ar@<-2ex>@{}[u]|{\rotatebox{90}{$\subset$}}
&\D_{lf}^{\uparrow}(\GrMod E): G_{E}. \ar@<0.5ex>[l] \ar@<3.5ex>@{}[u]|{\rotatebox{90}{$\subset$}}
}\]
such that $F_A(X(i)[j])=F_A(X)(-i)[i+j]$, $F_A(A_0)=E$ and $F_A(DA)= E_0$.
\end{thm} 

Note that $\D^b(\GrMod A) \subset \D^{\uparrow}_{lf}(\GrMod A)$.
However, $\D^b(\GrMod A) \not\subset \D^{\downarrow}(\GrMod A)$ unless $A$ is a finite-dimensional algebra, so the above equivalences are sometimes inconvenient to use. 
For this reason, we consider the following functors introduced by Mori \cite{Mo}.

\begin{dfn}[{\cite[Definition 2.5]{Mo}}]
Let $A$ be a Koszul algebra. We define the functors $\mathfrak{F}_A, \mathfrak{G}_A$ by the compositions
\begin{align*}
&\mathfrak{F}_A:\xymatrix @R=1pc@C=2pc{
 \D^{\uparrow}_{lf}(\GrMod A) \ar[r]^-{D}
&\D^{\downarrow}_{lf}(\GrMod  A^{\op})\ar[r]^-{F_{A^{\op}}}
&\D^{\uparrow}_{lf}(\GrMod  E^{\op})},\\
&\mathfrak{G}_A:\xymatrix @R=1pc@C=2pc{
\D^{\uparrow}_{lf}(\GrMod  A) \ar[r]^-{G_{A}}
&\D^{\downarrow}_{lf}(\GrMod  E) \ar[r]^-{D}
&\D^{\uparrow}_{lf}(\GrMod  E^{\op})},
\end{align*}
where $D$ is the graded $k$-dual.
\end{dfn}

Theorem \ref{t.BGS} induces the following.

\begin{thm}\label{t.KD}
Let $A$ be a Koszul algebra. Then 
\[
\mathfrak{F}_A: \D^{\uparrow}_{lf}(\GrMod A) \leftrightarrows \D^{\uparrow}_{lf}(\GrMod A^!): \mathfrak{G}_{A^!}
\]
form a duality such that
$\mathfrak{F}_A(X(i)[j])=\mathfrak{F}_A(X)(i)[-i-j]$, $\mathfrak{F}_A(A_0)=A^!$ and $\mathfrak{F}_A(A)\cong A^!_0$.
 \end{thm}

We will show that if $A$ and $A^!$ are Koszul AS-Gorenstein algebras, then $\mathfrak{F}_A$ and $\mathfrak{G}_{A^!}$ induce a duality between $\D^b(\grmod A)$ and  $\D^b(\grmod A^!)$.
To prove this, we make some preparations.

\begin{lem}\label{l.fg}
Let $A$ be a Koszul algebra.
\begin{enumerate}
\item  If $X,Y\in \D_{lf}^{\uparrow}(\GrMod A)$, then
$\Ext_A^\ell(X,Y)_m \cong \Ext_{A^!}^{\ell+m}(\mathfrak{F}_A(Y),\mathfrak{F}_A(X))_{-m}$.
\item If $X \in \D^b_{lf}(\GrMod A)$, then $\mathfrak{F}_A(X)^\ell \in \grmod A^!$ for any $\ell \in \ZZ$.
\item If $X \in \D^b_{lf}(\GrMod A)$, then $\mathfrak{G}_A(X)^\ell \in \grmod A^!$ for any $\ell \in \ZZ$.
\end{enumerate}
\end{lem}

\begin{proof}  
(1) This easily follows from Theorem \ref{t.KD}.

(2) Since we have
\begin{align*}
\mathfrak{F}_A(X)^\ell = 
F_{A^{\op}}(DX)^\ell=\bigoplus_{j}D(X)^{\ell-j}_{j}\otimes_{A^{\op}_0}E^{\op}(j) \cong \bigoplus_{j}D(X^{-\ell+j}_{-j})\otimes_{A_0}A^!(j),
\end{align*}
the assertion follows.

(3) Since we have
\begin{align*}
\mathfrak{G}_A(X)^\ell &=
D({G_A}(X))^\ell\cong D({G_A}(X)^{-\ell})
\cong D(\bigoplus_{j}\Hom_{A_0}(E(-j),X^{-\ell-j}_{j}))\\
&\cong \bigoplus_{j}D(\Hom_{A_0}(E(-j),X^{-\ell-j}_{j}))
\cong \bigoplus_{j}E(-j) \otimes_{A_0} D(X^{-\ell-j}_{j}),
\end{align*}
the assertion follows.
\end{proof}

\begin{lem}\label{l.hh}
Let $A$ be an AS-Gorenstein algebra of dimension $d$.
\begin{enumerate}
\item For $X \in \D^{-}(\GrMod A)$ and $Y \in \D^b(\GrMod A^{\en})$, we have 
$\RGamma_{\fm}(X\Lotimes_A Y) \cong X\Lotimes_A \RGamma_{\fm}(Y)$.
\item For $X,Y \in \D^b(\grmod A)$,  we have $\RHom_A(\RGamma_{\fm}(X), Y) \cong \RHom_A(X, Y)$.
\end{enumerate}
\end{lem}

\begin{proof}
(1) This follows from \cite[Proposition 2.1]{JoL}.

(2) First, we have
\begin{equation}\label{e.hyp1}
\begin{split}
\RHom_A(\RGamma_{\fm^{\op}}(A), Y)
&\cong \RHom_{A^{\op}}(DY, D\RGamma_{\fm^{\op}}(A))\\ 
&\overset{\text{Prop.~\ref{p.canm}(3)}}{\cong} D\RGamma_{\fm^{\op}}(DY)
\cong DDY \cong Y.
\end{split}
\end{equation}
Using this, we obtain
\begin{align*}
\RHom_A(\RGamma_{\fm}(X), Y)
&\overset{\text{(1)}}{\cong} \RHom_A(X\Lotimes_A \RGamma_{\fm}(A), Y)
\cong \RHom_A(X,\RHom_A(\RGamma_{\fm}(A), Y))\\
&\cong \RHom_A(X,\RHom_A(\RGamma_{\fm^{\op}}(A), Y))
\overset{\text{\eqref{e.hyp1}}}{\cong}\RHom_A(X,Y).\qedhere
\end{align*}
\end{proof}

Here we focus on the two notions of regularity.

\begin{dfn} Let $A$ be an $\NN$-graded algebra.
\begin{enumerate}
\item The \emph{Castelnuovo-Mumford regularity} of $X\in \D(\GrMod A)$ is defined by
\[ \CMreg X :=
\sup\{ s \mid \H_{\fm}^{j}(X)_{s-j}\neq 0\ \text{for some}\ j \in \ZZ \}.
\]
\item The \emph{Ext-regularity} of $X\in \D(\GrMod A)$ is defined by
\[ \Extreg X :=
-\inf\{ s \mid \Ext^{j}_A(X, A_0)_{s-j}\neq 0\ \text{for some}\ j \in \ZZ \}.
\]
\end{enumerate}
\end{dfn}
 
\begin{thm}\label{t.reg}
Let $A$ be a Koszul AS-Gorenstein algebra and let $0 \neq X \in \D^b(\grmod A)$. Then $\Extreg X \leq \CMreg X <\infty$.
\end{thm}

\begin{proof}
We follow the same strategy as in the proof of \cite[Theorem 2.5]{JoLF}.
By Proposition \ref{p.canm}(2),(3), we have
$D\RGamma_{\fm}(X) \in \D^b(\grmod A^{\op})$, so it follows that $\CMreg X \neq \infty$.
We now show that $\Extreg X \leq \CMreg X$.
Let 
\begin{align}\label{e.mprA0}
\xymatrix @R=1pc@C=1pc{
\cdots \ar[r] & P^i \ar[r]&\cdots \ar[r]&P^1 \ar[r]&P^0 \ar[r]&A_0 \ar[r] &0
}
\end{align}
be a graded minimal projective resolution of $A_0$ in $\GrMod A^{\op}$. Since $A^{\op}$ is Koszul by \cite[Proposition 2.9.1]{BGS}, $P^i$ is generated in degree $i$, so $P^i$ is a direct summand of $A(-i)^{\oplus s_i}$ for some $s_i$. Applying $D$ to \eqref{e.mprA0}, we obtain an exact sequence
\[
\xymatrix @R=1pc@C=1pc{
0\ar[r] &D(A_0) \ar[r] &D(P^0)\ar[r] &D(P^1)\ar[r] &\cdots \ar[r]& D(P^i) \ar[r] &\cdots.
}
\]
Since $A_0$ is semisimple, this gives an injective resolution of $A_0 \cong D(A_0)$  in $\GrMod A$. Let $r=\CMreg X$. Then
$\H_{\fm}^{-j}(X)_{>r+j}=0$ for any $j \in \ZZ$, so we obtain $D(\H_{\fm}^{-j}(X))_{<-r-j}=0$.
Since
\begin{align*}
\Hom_A(\H_{\fm}^{-j}(X), D(P^i))
\subset \Hom_A(\H_{\fm}^{-j}(X), D(A(-i)^{\oplus s_i}))
&\cong \Hom_A(\H_{\fm}^{-j}(X), DA)(i)^{\oplus s_i}\\
&\cong \Hom_{A^{\op}}(A, D(\H_{\fm}^{-j}(X)))(i)^{\oplus s_i}\\
&\cong D(\H_{\fm}^{-j}(X))(i)^{\oplus s_i},
\end{align*}
we have $\Hom_A(\H_{\fm}^{-j}(X), D(P^i))_{<-r-i-j}=0$.
Moreover, since $\Ext^i_A(\H_{\fm}^{-j}(X), A_0)$ is a subquotient of $\Hom_A(\H_{\fm}^{-j}(X), D(P^i))$, we obtain $\Ext^i_A(\H_{\fm}^{-j}(X), A_0)_{<-r-i-j}=0$.
By Lemma \ref{l.hh}(2), there exists a spectral sequence
\[
E^{i,j}_2 =\Ext^i_A(\H_{\fm}^{-j}(X),A_0) \Rightarrow
\Ext^{i+j}_A(X,A_0).
\]
Since $(E^{i,j}_2)_{<-r-(i+j)}=0$ for all $i,j \in \ZZ$, it follows that $\Ext^{\ell}_A(X,A_0)_{<-r-\ell}=0$ for all $\ell \in \ZZ$. This implies $-r\leq-\Extreg X$. Hence $\Extreg X\leq r=\CMreg X$.
\end{proof}

Thanks to Theorem \ref{t.reg}, we have the following result.

\begin{thm}\label{t.DbfgF}
Let $A$ be a Koszul AS-Gorenstein algebra.
If $X \in \D^b(\grmod A)$, then $\mathfrak{F}_A(X) \in \D^b(\grmod A^!)$.
Moreover, if $A^!$ is AS-Gorenstein and
$Y \in \D^b(\grmod A^!)$, then $\mathfrak{G}_{A^!}(Y) \in \D^b(\grmod A)$.
\end{thm}

\begin{proof}
If $X=0$, then the statement is clear, so we may assume that $X \neq 0$. 
Since $X \in \D^b(\grmod A) \subset  \D^{\uparrow}_{lf}(\grmod A)$, Theorem \ref{t.KD} and Lemma \ref{l.fg}(2) imply $\mathfrak{F}_A(X) \in \D^{\uparrow}(\grmod A^!)$.
In particular, $\sup \mathfrak{F}_A(X) <\infty$.
Moreover we have 
\begin{align*}
\inf \mathfrak{F}_A(X)
&= \inf \RHom_{A^!}(A^!, \mathfrak{F}_A(X))
=\inf\{ s \mid   \Ext^s_{A^!}(\mathfrak{F}_A(A_0), \mathfrak{F}_A(X))_{-s+j}\neq 0 \; \text{for some } j \in \ZZ \}\\
&\overset{\text{Lem.\,\ref{l.fg}(1)}}{=}\inf\{ s \mid   \Ext^{j}_A(X, A_0)_{s-j}\neq 0 \; \text{for some } j \in \ZZ \}
=-\Extreg X
\overset{\text{Thm.\,\ref{t.reg}}}{>}-\infty.
\end{align*}
Therefore $\mathfrak{F}_A(X) \in \D^b(\grmod A^!)$.
The proof of the second claim is similar.
\end{proof}

\begin{thm} \label{t.KDfg}
Let $A$ be a Koszul AS-Gorenstein algebra.
Assume that $A^!$ is AS-Gorenstein.
Then 
\[
\mathfrak{F}_A: \D^b(\grmod A) \leftrightarrows \D^b(\grmod A^!): \mathfrak{G}_{A^!}
\]
form a duality such that
$\mathfrak{F}_A(X(i)[j])=\mathfrak{F}_A(X)(i)[-i-j]$, $\mathfrak{F}_A(A_0)=A^!$ and $\mathfrak{F}_A(A)= A^!_0$.
\end{thm}

\begin{proof}
This follows from Theorems \ref{t.KD} and \ref{t.DbfgF}.
\end{proof}

The following theorem is the main result of this section, which yields Theorem \ref{t.qcdintro}.

\begin{thm} \label{t.qcd}
Let $A$ be a Koszul AS-Gorenstein algebra.
Assume that $A^!$ is AS-Gorenstein.
Then there exists a duality
\[
\overline{\mathfrak{F}}_A: \D^b(\grmod A)/\thick(\grproj  A) \leftrightarrows  \D^b(\grmod  A^!)/\thick(\fdim A^!): \overline{\mathfrak{G}}_{A^!}.
\]
In particular, this duality induces a duality $\mathfrak{E}:\uCM(A) \to \D^b(\qgr A^!)$.
\end{thm}

\begin{proof}
Since $\mathfrak{F}_A(A(j))
=A^!_0(j)[-j] \in \thick\{A^!_0(i) \mid i \in \ZZ \} \subset \D^b(\grmod A^!)$
and $\mathfrak{G}_{A^!}(A^!_0(j))
=A(j)[-j]\in \thick\{A(i) \mid i \in \ZZ \} \subset \D^b(\grmod A)$ by Theorem \ref{t.KDfg}, $\mathfrak{F}_A$ and $\mathfrak{G}_{A^!}$ restrict to a duality
\begin{align*}
\thick(\grproj A)=\thick\{A(i) \mid i \in \ZZ \} \leftrightarrows \thick\{A_0^!(i) \mid i \in \ZZ \}=\thick(\fdim A^!).
\end{align*}
Therefore, they induce a duality between the corresponding Verdier quotient categories.
Using the Buchweitz equivalence \cite{Bu} and \cite[Theorem~4.4]{MS}, we obtain a duality
\[
\mathfrak{E}:\uCM(A) \xrightarrow{\sim} 
\D^b(\grmod A)/\thick(\grproj A)
\xrightarrow{\ \overline{\mathfrak{F}}_A\ }
\D^b(\grmod A^!)/\thick(\fdim A^!)
\xrightarrow{\sim}
\D^b(\qgr A^!). \qedhere
\]
\end{proof}

We will use the following results later.

\begin{lem} \label{thickg}
Let $A$ be a Koszul AS-Gorenstein algebra of dimension $d$. Assume that $A^!$ is AS-Gorenstein.
Let $\mathfrak{E}:\uCM(A) \to \D^b(\qgr A^!)$ be the duality in Theorem \ref{t.qcd}.
Then $\mathfrak{E}(\Omega^d A_0(d)) \cong A^!(d)$.
\end{lem}

\begin{proof}
By Theorem \ref{t.KDfg}, we have $\mathfrak{F}_A(A_0(d)[-d])\cong A^!(d)$.
Moreover, $A_0(d) \cong \Omega^dA_0(d)[d]$ in $\D^b(\grmod A)/\thick(\grproj  A)$.
Hence $\overline{\mathfrak{F}}_A(\Omega^dA_0(d))\cong \overline{\mathfrak{F}}_A(A_0(d)[-d]) \cong A^!(d)$,
and the assertion follows.
\end{proof}

\begin{prop}\label{pfds}
Let $A$ be a Koszul AS-regular algebra.
Then $A^!$ is a finite-dimensional Koszul AS-Gorenstein algebra of dimension $0$.
\end{prop}

\begin{proof}
Let $d = \gldim A$. Since $A$ is Koszul AS-regular, we see that $A^!$ is finite-dimensional, 
$\Ext_A^i(A_0, A) = 0$ for $i \ne d$, and $\Ext_A^d(A_0, A)_m = 0$ for $m \ne -d$.
By Lemma \ref{l.fg}, we have
$\Ext_{A^!}^i(A^!_0, A^!)_m\cong \Ext_{A^!}^i(\mathfrak{F}_A(A), \mathfrak{F}_A(A_0))_m
\cong \Ext_{A}^{i+m}(A_0, A)_{-m}$.
Therefore, $\Ext_{A^!}^i(A^!_0, A^!) = 0$ for all $i \ne 0$, 
and hence $A^!$ is self-injective.
Since $\Hom_{A^!}(N,A^!) \neq 0$ for every graded simple module $N\in \grmod A^!$, and $\dim_k \Hom_{A^!}(A_0^!,A^!)=\dim_k \Ext^d_{A}(A_0,A)=\dim_k A_0$, 
it follows that $\dim_k \Hom_{A^!}(N,A^!)=1$ for every graded simple module $N\in \grmod A^!$.
Hence $A^!$ is AS-Gorenstein of dimension $0$.
\end{proof}

\section{Trivial extensions}

In this section, we study trivial extensions of graded algebras 
and relate them to Ore extensions and Zhang twists.

\begin{dfn}
Let $A$ be a $\mathbb Z$-graded algebra and
let $M \in \GrMod A^{\en}$ be a graded $A$-bimodule.
The graded trivial extension of $A$ by $M$ is the $\ZZ$-graded algebra
defined by
\[
A \ltimes M := \bigoplus_{i \in\mathbb Z} (A_i \oplus M_i),
\]
with multiplication
\[
(a,m)(a',m') = (aa',\, am' + ma')
\]
for $a,a'\in A$ and $m,m'\in M$.
\end{dfn}

\begin{lem}\label{lt1}
Let $A$ be a $\ZZ$-graded algebra and let  $\sigma\in\GrAut(A)$.
For an integer $s \in \mathbb Z$, consider the graded $A$-bimodule
$L := A_\sigma(-s)$ in $\GrMod A^{\en}$.
Then there is an isomorphism of graded algebras $A \ltimes L \cong A[x;\sigma]/(x^2)$, where $\deg x = s$.
\end{lem}

\begin{proof}
Define a $k$-linear map
$\Phi \colon A \ltimes L \to A[x;\sigma]/(x^2)$
by $\Phi(a,b)=a+bx$ for $a\in A$ and $b\in L$.
Then $\Phi$ is graded and bijective since both algebras are isomorphic to
$A \oplus A(-s)$ as graded vector spaces. For $a,a'\in A$ and $b,b'\in L$, we have
\[\Phi((a,b)(a',b'))
= \Phi(aa',\, a\cdot b' + b\cdot a')
= \Phi(aa',\, ab' + b\sigma(a'))
= aa' + (ab' + b \sigma(a'))x.
\]
On the other hand,
\[
\Phi(a,b)\Phi(a',b')
= (a + bx)(a' + b'x)
= aa' + ab'x + bxa' + bxb'x
= aa' + (ab' + b\si(a'))x
\]
by the Ore relation and $x^2=0$.
Hence $\Phi$ is an isomorphism of $\mathbb Z$-graded algebras.
\end{proof}

\begin{lem}\label{ltg}
Let $A$ be an AS-Gorenstein algebra of dimension $d$ and let $\sigma\in\GrAut(A)$.
For an integer $s \geq 1$, consider the graded $A$-bimodule
$L := A_\sigma(-s)$ in $\GrMod A^{\en}$.
Then $A \ltimes L$ is AS-Gorenstein of dimension $d$.
\end{lem}

\begin{proof}
By Lemma \ref{lt1}, we have an isomorphism
$A \ltimes L \cong A[x;\sigma]/(x^2)$ with $\deg x = s\geq 1$.
By Proposition \ref{poz}(1), $A[x;\sigma]$
is AS-Gorenstein of dimension $d+1$.
Since $x^2$ is also a regular normal element of $A[x;\sigma]$,
Proposition~\ref{pLe} implies that
$A[x;\sigma]/(x^2)$ is AS-Gorenstein of dimension $d$.
\end{proof}

We remark that when $A$ is connected graded and $L=\om_A(-s)\cong A_\nu(-\ell-s)$ with $s\geq -\ell+1$, the assertion of Lemma \ref{ltg} also follows from \cite[Proposition 1.5]{JoC}.

\begin{lem}\label{ltt}
Let $A$ be a $\mathbb Z$-graded algebra and let $\sigma \in \GrAut(A)$.
Consider the graded algebra $A[x;\sigma]/(x^2)$ with $\deg x=1$.
\begin{enumerate}
\item
The assignment $\widehat{\sigma}(a+bx)=\sigma(a)+\sigma(b)x$ for $a,b\in A$
defines a graded algebra automorphism of $A[x;\sigma]/(x^2)$.
\item
$(A[x;\sigma]/(x^2))^{\widehat{\sigma}^{-1}} \cong A^{\sigma^{-1}}[x]/(x^2)$ as graded algebras.
\end{enumerate}
\end{lem}

\begin{proof}
(1) For $a+bx, a'+b'x\in A[x;\sigma]/(x^2)$, we have
\[
\widehat{\sigma}((a+bx)(a'+b'x))
=\si(aa')+\si(ab'+b\si(a'))x
=\widehat{\sigma}(a+bx)\widehat{\sigma}(a'+b'x),
\]
so $\widehat{\sigma}$ defines a graded algebra automorphism.

(2) For $a+bx \in (A[x;\sigma]/(x^2))^{\widehat{\sigma}^{-1}}_i$
and $a'+b'x \in (A[x;\sigma]/(x^2))^{\widehat{\sigma}^{-1}}_j$,
the multiplication is given by
\begin{align*}
(a+bx)*(a'+b'x)&=(a+bx)\widehat{\sigma}^{-i}(a'+b'x)\\
&=(a+bx)(\si^{-i}(a')+\si^{-i}(b')x)
=a\si^{-i}(a')+(a\si^{-i}(b')+b\si^{-i+1}(a'))x.
\end{align*}
On the other hand, if $a+bx \in (A^{\sigma^{-1}}[x]/(x^2))_i$
and $a'+b'x \in (A^{\sigma^{-1}}[x]/(x^2))_j$,
then $a\in A^{\sigma^{-1}}_i$ and $b\in A^{\sigma^{-1}}_{i-1}$, so the multiplication is
\begin{align*}
(a+bx)(a'+b'x)&=a*a'+(a*b'+b*a')x=a\si^{-i}(a')+(a\si^{-i}(b')+b\si^{-(i-1)}(a'))x.
\end{align*} 
Thus the two graded algebras have the same multiplication, and hence are isomorphic.
\end{proof}

\begin{lem}\label{lts}
Let $A$ be an AS-Gorenstein algebra and let $\sigma\in\GrAut(A)$. Consider the graded $A$-bimodule
$L := A_\sigma(-1)$ in $\GrMod A^{\en}$.
Then there is an equivalence of categories
$\GrMod(A \ltimes L) \simeq \GrMod(A^{\sigma^{-1}}[x]/(x^2))$, where $\deg x=1$.
In particular, we have
$\qgr (A \ltimes L) \simeq \qgr(A^{\sigma^{-1}}[x]/(x^2))$,
$\CM(A \ltimes L)\simeq \CM(A^{\sigma^{-1}}[x]/(x^2))$, and  $\uCM(A \ltimes L)\simeq \uCM(A^{\sigma^{-1}}[x]/(x^2))$.
\end{lem}

\begin{proof}
By Lemmas \ref{lt1}, \ref{ltg},
$A \ltimes L \cong A[x;\sigma]/(x^2)$ with $\deg x=1$ is an AS-Gorenstein algebra.
By Lemma \ref{ltt}(2), $A^{\sigma^{-1}}[x]/(x^2)$ is isomorphic to a Zhang twist of $A[x;\sigma]/(x^2)$.
Hence Proposition \ref{poz}(3) implies that $A^{\sigma^{-1}}[x]/(x^2)$ is AS-Gorenstein.
Moreover, Theorem \ref{tZ} yields an equivalence $\GrMod(A \ltimes L) \xrightarrow{\sim} \GrMod(A^{\sigma^{-1}}[x]/(x^2))$, which induces equivalences on $\qgr$, $\CM$, and $\uCM$.
\end{proof}

\section{Main results}

In this section, we prove Theorems \ref{tmain1} and \ref{tmain2}, and derive some corollaries.

\subsection{Proof of Theorem \ref{tmain1}}

We first recall basic results on graded localization.

\begin{lem}\label{lloc}
\begin{enumerate}
\item Let $A$ be a noetherian $\NN$-graded algebra
generated by $A_1$ over $A_0$.
Let $g \in A$ be a homogeneous regular normal element of positive degree. Assume that $\dim_k A/(g)<\infty$.
Then $\grmod A \to \mod A[g^{-1}]_0; M \to M[g^{-1}]_0$ induces an equivalence $\qgr A \to \mod A[g^{-1}]_0$.
\item Let $A$ be an $\ZZ$-graded algebra and let $g\in A$ be a homogeneous regular normal element.
For a positive integer $m$, we have $A[(g^m)^{-1}]=A[g^{-1}]$ as graded algebras. In particular, $A[(g^m)^{-1}]_0=A[g^{-1}]_0$.
\end{enumerate}
\end{lem}

\begin{proof}
(1) The proof is essentially the same as that of \cite[Proposition 4.6]{MUk}.

(2) The proof is straightforward.
\end{proof}

Moreover, we need the following result.

\begin{lem}\label{lqp}
Let $A$ be an AS-Gorenstein algebra of dimension $1$ generated by $A_1$ over $A_0$. Then $\thick(A(s)) =\thick(A)$ in $\D^b(\qgr A)$ for all $s \in \ZZ$.
\end{lem}

\begin{proof}
Since $\qgr A$ is a Krull-Schmidt category,
for an object $X\in \qgr A$, let $\Ind(X)$ denote the set of isomorphism classes of indecomposable direct summands of $X$ in $\qgr A$.

Since $A$ is generated in degree $1$, there is an exact sequence
$A(-1)^{\oplus n} \to A \to A_0 \to 0$
in $\grmod A$. This induces an exact sequence
\begin{equation}\label{eqgr}
A(-1)^{\oplus n} \to A \to 0
\end{equation}
in $\qgr A$.
Since $A$ is AS-Gorenstein of dimension $1$, it follows from \cite[Proposition 3.5(1)]{IKU} that $A$ is a projective object in $\qgr A$.
Therefore the exact sequence \eqref{eqgr} splits, and hence $A$ is a direct summand of $A(-1)^{\oplus n}$ in $\qgr A$. Thus we have $\Ind(A) \subset \Ind(A(-1))$.
Since the shift functor $(1): \qgr A \to \qgr A$ is an autoequivalence, it follows that $|\Ind(A)| = |\Ind(A(-1))|$, so we obtain $\Ind(A) = \Ind(A(-1))$.
Iterating this argument, we have $\Ind(A) = \Ind(A(s))$ for any $s \in \ZZ$. Hence $\thick(A)=\thick(A(s))$ in $\D^b(\qgr A)$.
\end{proof}

The following theorem plays a central role in proving Theorem \ref{tmain1}.

\begin{thm}\label{tNSV}
Let $S$ be a Koszul AS-regular algebra, and let $f \in S$ be a homogeneous regular normal element of degree $2$. Set $A=S/(f)$, which is a Koszul AS-Gorenstein algebra.
If $g \in A^!$ is a homogeneous regular normal element of degree $2$ such that $A^!/(g) \cong S^!$, then there is an equivalence of triangulated categories
\[
\uCM(A) \simeq \D^b(\mod (A^![g^{-1}]_0)^{\op}).
\]
Moreover, this equivalence restricts to a triangle equivalence
\[
\thick(\Omega^d A_0(d)) \xrightarrow{\sim}
\K^b(\proj (A^![g^{-1}]_0)^{\op}).
\]
\end{thm}

\begin{proof}
Set $\Lambda = A^![g^{-1}]_0$.
Since $S$ is Koszul AS-regular, $A^!/(g) \cong S^!$ is finite-dimensional by Proposition \ref{pfds}.
Therefore, Proposition \ref{lloc}(1) yields an equivalence
$\qgr A^! \to \mod \La$, and hence induces a triangle equivalence
\[ L:\D^b(\qgr A^!) \to \D^b(\mod \La). \]
Moreover, $A^!/(g) \cong S^!$ is an AS-Gorenstein algebra of dimension $0$ by Proposition \ref{pfds}, so $A^!$ is a Koszul AS-Gorenstein algebra of dimension $1$ by Proposition \ref{pLe}. 
Thus, Theorem \ref{t.qcd} provides a duality
\[ \mathfrak{E}:\uCM(A) \to \D^b(\qgr A^!). \]
Consequently, we obtain an equivalence of triangulated categories
\begin{equation*}
\uCM(A)
\xrightarrow[\mathfrak{E}]{\text{duality}}
\D^b(\qgr A^!)
\xrightarrow[L]{\sim}
\D^b(\mod  \La)
\xrightarrow[D]{\text{duality}}
\D^b(\mod \La^{\op}),
\end{equation*}
where $D$ denotes the $k$-dual.
This proves the first statement.

By Lemma \ref{thickg}, we have $\mathfrak{E}(\Omega^d A_0(d)) \cong A^!(d)$. By Lemma \ref{lloc}(1), we have $L(A^!) \cong \La$ in $\D^b(\mod \La)$.
Since $\La$ is self-injective by Proposition \ref{pfds}, it follows that $D\La \cong \La$ in $\mod \La^{\op}$.
Therefore, we obtain
\[
\thick(\Omega^d A_0(d))
\xrightarrow[\mathfrak{E}]{}
\thick(A^!(d)) \overset{\text{Lem.\,\ref{lqp}}}{=}
\thick(A^!) 
\xrightarrow[L]{\sim}
\thick(\La) 
\xrightarrow[D]{}
\thick(_{\La}\La)=\K^b(\proj \La^{\op}).
\]
This completes the proof.
\end{proof}

For a $\ZZ$-graded algebra $A$, we define a graded algebra automorphism $-1 \in \GrAut(A)$ by $a \mapsto (-1)^{\deg a} a$
for all homogeneous elements $a \in A$.
For example, $k[y,z][x;-1] \cong  k\ang{y, z, x}/(yz-zy,\,xy+yx,\,xz+zx)$, where $\deg x=\deg y=\deg z=1$.
The following lemma provides a key computational step in the proof of Theorem \ref{tmain1}.

\begin{lem}\label{lCa}
Let $S$ be a Koszul AS-regular algebra. 
\begin{enumerate}
\item $S[x]$ with $\deg x=1$ is a Koszul AS-regular algebra. 
\item $A=S[x]/(x^2)$ with $\deg x=1$ is a Koszul AS-Gorenstein algebra.
\item
There are isomorphisms of graded algebras $A^! \cong S^![z;-1]$ and $S[x]^! \cong S^![z;-1]/(z^2)$.
\item
Let $g=z^2 \in S^![z;-1]_2 \cong  A^!_2$.
Then $g$ is a regular central element of $A^!$ and $A^!/(g)\cong S[x]^!$.
\item $A^![g^{-1}]_0$ is isomorphic to $S^!$ as (ungraded) algebras.
\end{enumerate}
\end{lem}

\begin{proof}
(1) This follows from Proposition \ref{poz}(1), (2).

(2) By Proposition \ref{pLe}, $A$ is an AS-Gorenstein algebra. The same argument as in \cite[Theorem 1.2]{ST} implies that $A$ is Koszul.

(3) Set $S=T_{S_0}(S_1)/(R_S)$. Then $A \cong T_{S_0}(S_1\oplus V)/(R_S,\ S_1\otimes_{S_0} V-V\otimes_{S_0} S_1,\ V\otimes_{S_0} V)$, where $V=S_0x$. Therefore $A^!\cong T_{S_0}(S_1^*\oplus V^*)/(R_S^\perp,\ S_1^*\otimes_{S_0} V^*+V^*\otimes_{S_0} S_1^*) \cong S^![z;-1]$. Similarly, one obtains an isomorphism $S[x]^! \cong S^![z;-1]/(z^2)$.

(4) This is an immediate consequence of (3).

(5) By (3), (4), and Lemma \ref{lloc}(2), we obtain
$A^![g^{-1}]_0 \cong (S^![z;-1])[(z^2)^{-1}]_0
\cong (S^![z;-1])[z^{-1}]_0$.
Define a map $\Psi \colon S^! \to (S^![z;-1])[z^{-1}]_0$ by
\[
\Psi(a)=(-1)^{\frac{\deg a(\deg a-1)}{2}}a z^{-\deg a}
\quad \text{for all homogeneous } a \in S^!.
\]
Then,  for $a \in S^!_p, b \in S^!_q$, we have
\begin{align*}
\Psi(a)\Psi(b)&= (-1)^{\frac{p(p-1)}{2}}(-1)^{\frac{q(q-1)}{2}}a z^{-p}bz^{-q}
=(-1)^{\frac{p(p-1)}{2} + \frac{q(q-1)}{2}}(-1)^{pq}abz^{-(p+q)}\\
&=(-1)^{\frac{(p+q)(p+q-1)}{2}}abz^{-(p+q)}=\Psi(ab),
\end{align*}
so $\Psi$ is an algebra homomorphism. It is clear that $\Psi$ is bijective.
Therefore $\Psi$ is an algebra isomorphism, and the result follows.
\end{proof}

We can now prove Theorem \ref{tmain1}.

\begin{thm}[Theorem \ref{tmain1}]\label{tmain1a}
Let $S$ be a Koszul AS-regular algebra and let $\sigma\in\GrAut(S)$. Consider the graded $S$-bimodule
$L := S_\sigma(-1)$ in $\GrMod S^{\en}$.
Then there are equivalences of triangulated categories
\begin{equation*}
\xymatrix@R=0.5cm@C=1.25cm{
\uCM(S \ltimes L)\ar[r]^-{\sim} 
& \D^b(\mod ((S^{\sigma^{-1}})^!)^{\op})\\
\thick(\Omega^d (S \ltimes L)_0(d)) \ar[r]^-{\sim} \ar@{}[u]|{\rotatebox{90}{$\subset$}}
& \K^b(\proj ((S^{\sigma^{-1}})^!)^{\op}). \ar@<2ex>@{}[u]|{\rotatebox{90}{$\subset$}}
}
\end{equation*}
In particular, $\thick(\Omega^d (S \ltimes L)_0(d))$ admits a Serre functor and Auslander-Reiten triangles.
\end{thm}

\begin{proof}
Set $A =S^{\sigma^{-1}}[x]/(x^2)$ with $\deg x=1$. By Lemma \ref{lts}, 
there is an equivalence
\[
F: \uCM(S \ltimes L)\xrightarrow{\sim} \uCM(A)
\]
such that $F(\Omega^d(S \ltimes L)_0(d)) \cong \Omega^d A_0(d)$.
It follows from Proposition \ref{poz} that $S^{\sigma^{-1}}$ and $S^{\sigma^{-1}}[x]$ are Koszul AS-regular algebras.
Thus, by Lemma \ref{lCa}, there is a regular central element $g \in A^!_2$ such that
$A^!/(g) \cong S^{\sigma^{-1}}[x]^!$ (as graded algebras) and $A^![g^{-1}]_0 \cong (S^{\sigma^{-1}})^!$ (as ungraded algebras).
Hence we obtain the desired equivalences
\begin{equation*}
\xymatrix@R=0.5cm@C=2cm{
\uCM(S \ltimes L)\ar[r]^-{\sim} _-{F}
& \uCM(A) \ar[r]^-{\sim}_-{\text{Thm.\,\ref{tNSV}}}
& \D^b(\mod ((S^{\sigma^{-1}})^!)^{\op})\\
\thick(\Omega^d (S \ltimes L)_0(d)) \ar[r]^-{\sim}_-{F} \ar@{}[u]|{\rotatebox{90}{$\subset$}}
&\thick(\Omega^d A_0(d)) \ar[r]^-{\sim}_-{\text{Thm.\,\ref{tNSV}}}
\ar@{}[u]|{\rotatebox{90}{$\subset$}}
& \K^b(\proj ((S^{\sigma^{-1}})^!)^{\op}). \ar@<2ex>@{}[u]|{\rotatebox{90}{$\subset$}}
}
\end{equation*}

Since $(S^{\sigma^{-1}})^!$ is self-injective,
$\K^b(\proj ((S^{\sigma^{-1}})^!)^{\op})$ has a Serre functor and Auslander-Reiten triangles by \cite[Theorem 3.4]{Hap2} and \cite[Theorem I.2.4]{RV} (see also \cite[Proposition 4.5]{BIY}).
Therefore, the last statement of the theorem follows.
\end{proof}

If $S$ is a connected graded Koszul AS-regular algebra of dimension $d$, then its canonical module is given by $\omega_S \cong S_{\nu}(-d)$,
where $\nu$ is the Nakayama automorphism.
Thus, the following result follows from Theorem \ref{tmain1}.

\begin{cor}\label{cmain1}
Let $S$ be a connected graded Koszul AS-regular algebra of dimension $d$.
Then there are equivalences of triangulated categories
\begin{equation*}
\xymatrix@R=0.5cm@C=1.25cm{
\uCM(S \ltimes \omega_S(d-1))\ar[r]^-{\sim} 
& \D^b(\mod ((S^{\nu^{-1}})^!)^{\op})\\
\thick(\Omega^d k(d)) \ar[r]^-{\sim} \ar@{}[u]|{\rotatebox{90}{$\subset$}}
& \K^b(\proj ((S^{\nu^{-1}})^!)^{\op}). \ar@<2ex>@{}[u]|{\rotatebox{90}{$\subset$}}
}
\end{equation*}
In particular, $\thick(\Omega^d k(d))$ admits a Serre functor and Auslander-Reiten triangles.
\end{cor}

\subsection{Proof of Theorem \ref{tmain2}}
Before proving Theorem \ref{tmain2}, we present three lemmas.

\begin{lem}\label{lcolo}
Let $A$ be a finite-dimensional connected graded algebra.
Then the underlying (ungraded) algebra $A$ is a local algebra
with unique maximal ideal $\fm=A_{\geq 1}$.
\end{lem}

\begin{proof}
Clearly $\fm$ is an ideal of $A$.
Since $A$ is finite-dimensional over $k$, there exists an integer $N>0$
such that $A_i = 0$ for all $i>N$. Consequently, $\fm^{N+1}
\subset A_{\geq N+1}= 0$, so $\fm$ is a nilpotent ideal.
Thus it is contained in the Jacobson radical $\rad A$.
On the other hand, we have $A/\fm\cong k$, which is a field. It follows that $\fm$ is a maximal ideal of $A$, and hence $\fm=\rad A$. 
Therefore $A$ is a local algebra with unique maximal ideal $\fm$.
\end{proof}

\begin{lem}\label{llode}
Let $A$ and $B$ be finite-dimensional local algebras.
If $\D^b(\mod A) \simeq \D^b(\mod B)$, then $A \cong B$ as algebras.
\end{lem}

\begin{proof}
Since $A$ is a local algebra, the equivalence $\D^b(\mod A) \simeq \D^b(\mod B)$ implies that $A$ and $B$ are Morita equivalent
by \cite[Corollary~2.13]{RZ}.
Hence there exists a progenerator $P \in \mod A$ such that $\End_A(P) \cong B$.
Since $A$ is local, every finitely generated projective right $A$-module is free. Thus $P \cong A^{\oplus n}$ for some $n\geq 1$, and hence
$B\cong \End_A(A^{\oplus n}) \cong M_n(A)$.
Since $B$ is local, $M_n(A)$ must be local. This is possible only when $n=1$. Therefore $A\cong B$.
\end{proof}

\begin{lem}\label{lung}
Let $A$ and $B$ be finite-dimensional connected graded  algebras.
Assume that both $A$ and $B$ are generated in degree $1$.
Then $A$ and $B$ are isomorphic as graded algebras
if and only if they are isomorphic as (ungraded) algebras.
\end{lem}

\begin{proof}
The ``only if'' direction is clear.

For the converse, assume that there exists an isomorphism of ungraded
algebras $\vph: A \xrightarrow{\sim} B$.
By Lemma~\ref{lcolo}, the underlying algebras $A$ and $B$ are local algebras with unique maximal ideals $\fm_A = A_{\ge 1}$ and $\fm_B = B_{\ge 1}$, respectively.
Since $\vph$ is an algebra isomorphism, it preserves maximal ideals,
and hence $\vph(\fm_A) = \fm_B$.
It follows that $\vph(\fm_A^i) = \fm_B^i$ for all $i \geq 1$.
Therefore, $\vph$ induces an isomorphism of associated graded algebras with respect to the $\fm$-adic filtrations:
\[
\gr(\vph): \gr_{\fm_A} A
\xrightarrow{\sim}
\gr_{\fm_B} B.
\]
Since $A$ is generated in degree $1$, we have $\fm_A^i = A_{\ge i}$
for all $i \geq 1$, and similarly $\fm_B^i = B_{\ge i}$.
Hence the associated graded algebras
$\gr_{\fm_A} A$ and $\gr_{\fm_B} B$
coincide with $A$ and $B$, respectively, as graded algebras.
Via these identifications, the isomorphism $\gr(\varphi)$ yields
a graded algebra isomorphism $A \xrightarrow{\sim} B$.
This completes the proof.
\end{proof}

We now proceed to the proof of Theorem \ref{tmain2}.

\begin{thm}[Theorem \ref{tmain2}]\label{tmain2a}
Let $S$ and $R$ be connected graded Koszul AS-regular algebras,
and let $\sigma\in\GrAut(S)$ and $\tau\in\GrAut(R)$.
Consider the graded bimodules
$L := S_\sigma(-1)$ in $\GrMod S^{\en}$ and $N := R_\tau(-1)$ in $\GrMod R^{\en}$. Then the following conditions are equivalent:
\begin{enumerate}[label=(\roman*)]
\item $\uCM(S \ltimes L) \simeq \uCM(R \ltimes N)$ as triangulated categories.
\item $S^{\sigma^{-1}} \cong R^{\tau^{-1}}$ as graded algebras.
\end{enumerate}
In particular, if these conditions hold, then there is an equivalence
$\GrMod S \simeq \GrMod R$.
\end{thm}

\begin{proof}
By Proposition \ref{poz}(3), (4), the algebras $S^{\sigma^{-1}}$ and $R^{\tau^{-1}}$ are connected graded Koszul AS-regular.
Hence, by Proposition \ref{pfds}, their Koszul duals
$(S^{\sigma^{-1}})^!$ and $(R^{\tau^{-1}})^!$
are finite-dimensional connected graded algebras,
and therefore local as ungraded algebras by Lemma \ref{lcolo}.
We now obtain the following sequence of equivalences:
\begin{align*}
\uCM(S \ltimes L) \simeq \uCM(R \ltimes N)
\ \overset{\ \ \text{Thm.\,\ref{tmain1}}\ \ }{\Longleftrightarrow}
\ &\D^b(\mod ((S^{\sigma^{-1}})^!)^{\op})
\simeq
\D^b(\mod ((R^{\tau^{-1}})^!)^{\op}) \\
\ \overset{\text{\ \ $k$-duality\ \ }}{\Longleftrightarrow}
\ &\D^b(\mod ((S^{\sigma^{-1}})^!))
\simeq
\D^b(\mod ((R^{\tau^{-1}})^!)) \\
\ \overset{\ \ \text{Lem.\,\ref{llode}}\ \ }{\Longleftrightarrow}
\ &(S^{\sigma^{-1}})^! \cong (R^{\tau^{-1}})^!
\quad \text{(as ungraded algebras)} \\
\ \overset{\ \ \text{Lem.\,\ref{lung}}\ \  }{\Longleftrightarrow}
\ &(S^{\sigma^{-1}})^! \cong (R^{\tau^{-1}})^!
\quad \text{(as graded algebras)} \\
\ \overset{\text{duality $(-)^!$}}{\Longleftrightarrow}
\ &S^{\sigma^{-1}} \cong R^{\tau^{-1}}
\quad (\text{as graded algebras}).
\end{align*}

Moreover, if $S^{\sigma^{-1}} \cong R^{\tau^{-1}}$
as graded algebras, then Theorem~\ref{tZ} yields
\[
\GrMod S \simeq \GrMod S^{\sigma^{-1}} \simeq \GrMod R^{\tau^{-1}} \simeq \GrMod R.
\]
This proves the last assertion.
\end{proof}

Combining Theorem \ref{tmain2}  and Corollary \ref{cmain1},
we obtain the following.

\begin{cor}\label{cmain2}
Let $S$ and $R$ be Koszul AS-regular algebras
of dimensions $d$ and $d'$, respectively.
Let $\nu \in \GrAut(S)$ and $\mu \in \GrAut(R)$
be the Nakayama automorphisms.
Consider the following conditions:
\begin{enumerate}[label=(\alph*)]
\item $S \cong R$\ \textnormal{(as graded algebras)};
\item $\uCM(S \ltimes \omega_S(d-1))
\simeq
\uCM(R \ltimes \omega_R(d'-1))$;
\item $S^{\nu^{-1}} \cong R^{\mu^{-1}}$\ \textnormal{(as graded algebras)};
\item $\GrMod S \simeq \GrMod R$.
\end{enumerate}
Then the implications
$\textnormal{(a)} \Rightarrow
\textnormal{(b)} \Leftrightarrow
\textnormal{(c)} \Rightarrow
\textnormal{(d)}$
hold.
\end{cor}

\section{Examples}

We present several examples illustrating the main results.
In this section, we work over an algebraically closed field $k$ of characteristic $0$.

To illustrate Corollary \ref{cmain2}, we first compute the Zhang twist $S^{\nu^{-1}}$ of a connected graded Koszul AS-regular algebra $S$ by the inverse of its Nakayama automorphism $\nu$.

By \cite[Theorem 9.2]{Va}, the Nakayama automorphism $\nu$ of a connected graded Koszul AS-regular algebra $S$ is determined by that of its Koszul dual $S^!$, which is a Koszul self-injective algebra. More precisely,
$\nu(a)=(-1)^{(d+1)i}\eta'(a)$
for $a \in S_i$, where $\eta'$ denotes the automorphism of $S$ induced by the Nakayama automorphism of $S^!$.

Moreover, the Nakayama automorphism $\eta$ of a connected graded Koszul 
self-injective algebra $A$ can be computed using the Frobenius 
structure of $A$. Namely, if $n$ denotes the socle degree, then for 
$a\in A_1$ and $b\in A_{n-1}$ one has $ab=b\eta(a)$,
which determines $\eta$; see \cite[Lemma 3.3]{Sm}.

We begin with the case where $S$ is a $2$-dimensional connected graded Koszul AS-regular algebra.
It is well-known that every such algebra is isomorphic to
$k\langle x,y\rangle/(xy-\alpha yx)$ $(0\neq \alpha \in k)$ or $k\langle x,y\rangle/(xy-yx-x^2)$.

\begin{ex}\label{exntd2}
(1) We first treat the case 
$S = k\langle x,y\rangle/(xy-\alpha yx) \ (0\neq \alpha \in k)$.
It is easy to see that $S^! \cong k\langle x,y\rangle/(x^2, \al xy+yx, y^2)$.
Since $xy = -\alpha^{-1}yx$ in $S^!$, 
the Nakayama automorphism $\eta$ of $S^!$ is given by 
$\eta(x)=-\alpha^{-1}x$ and $\eta(y)=-\alpha y$.
Thus the Nakayama automorphism $\nu$ of $S$ is determined by
$\nu(x)=\alpha^{-1}x$ and $\nu(y)=\alpha y$, 
and its inverse is given by
\[ \nu^{-1}(x)=\alpha x \quad\text{and}\quad \nu^{-1}(y)=\alpha^{-1}y. \]
In $S^{\nu^{-1}}$, we have
\[
x \ast y= x\nu^{-1}(y)= \al^{-1}xy= yx =\al^{-1}y\nu^{-1}(x)= \al^{-1}y\ast x,
\] 
so it follows that  $S^{\nu^{-1}}\cong k\langle x,y\rangle/(xy-\alpha^{-1} yx)$.
Hence we obtain $S^{\nu^{-1}} \cong S$ as graded algebras.

(2) We next treat the case 
$S = k\langle x,y\rangle/(xy- yx-x^2)$.
It is easy to see that $S^! \cong k\langle x,y\rangle/(xy+x^2, xy+yx, y^2)$.
Since $x^2=x(-x-2y)$, $xy=y(-x-2y)$, and $yx=-xy$ in $S^!$, 
the Nakayama automorphism $\eta$ of $S^!$ is given by 
$\eta(x)=-x-2y$ and $\eta(y)=-y$.
Thus the Nakayama automorphism $\nu$ of $S$ is determined by
$\nu(x)=x$ and $\nu(y)=2x+y$, 
and its inverse is given by
\[ \nu^{-1}(x)=x \quad\text{and}\quad \nu^{-1}(y)=-2x+y. \]
In $S^{\nu^{-1}}$, we have
\[
x \ast y= x\nu^{-1}(y)= -2x^2 +xy= -x^2+yx =-x\nu^{-1}(x)+ y\nu^{-1}(x)=-x\ast x +y\ast x,
\]
so it follows that  $S^{\nu^{-1}}\cong k\langle x,y\rangle/(xy-yx+x^2)$.
Hence we obtain $S^{\nu^{-1}} \cong S$ as graded algebras.

Combining (1), (2), and Corollary~\ref{cmain2}, 
we conclude that for $2$-dimensional Koszul AS-regular algebras 
$S$ and $R$, 
\[
S \cong R\ \text{(as graded algebras)}
\quad \Longleftrightarrow \quad
\uCM(S \ltimes \omega_S(1))
\simeq
\uCM(R \ltimes \omega_R(1)).
\]
\end{ex}

We now turn to typical 
$3$-dimensional connected graded Koszul AS-regular algebras,
namely skew polynomial algebras in three variables.

\begin{ex}\label{exntd3}
Let
$S = k\langle x,y,z\rangle/(xy-\al_1 yx, yz-\al_2 zy, zx-\al_3 xz)$
be a standard graded skew polynomial algebra in three variables, where $\al_1\al_2\al_3\neq 0$.
One easily checks that $S^! \cong k\langle x,y,z\rangle/(\al_1 xy+yx, \al_2 yz+zy, \al_3 zx+xz, x^2, y^2, z^2)$.
Since 
$x(yz)=\al_1^{-1}\al_3(yz)x$, 
$y(zx)=\al_2^{-1}\al_1(zx)y$, and
$z(xy)=\al_3^{-1}\al_2(xy)z$
in $S^!$, the Nakayama automorphism $\eta$ of $S^!$ is given by 
$\eta(x)=\al_1^{-1}\al_3 x$,
$\eta(y)=\al_2^{-1}\al_1 y$, and
$\eta(z)=\al_3^{-1}\al_2 z$.
Thus the Nakayama automorphism $\nu$ of $S$ is given by
$\nu(x)=\al_1^{-1}\al_3 x$,
$\nu(y)=\al_2^{-1}\al_1 y$, and
$\nu(z)=\al_3^{-1}\al_2 z$, 
and its inverse is given by
\[ \nu^{-1}(x)=\al_1\al_3^{-1} x, \quad
\nu^{-1}(y)=\al_2\al_1^{-1} y,\ \ \text{and}\ \ 
\nu^{-1}(z)=\al_3\al_2^{-1} z.\]
In $S^{\nu^{-1}}$, we have
\[
x \ast y= x\nu^{-1}(y)= \al_2\al_1^{-1} xy=\al_2 yx =\al_1^{-1}\al_2\al_3 y\nu^{-1}(x)= \al_1^{-1}\al_2\al_3 y\ast x,
\]
and similarly
$y \ast z = \al_1\al_2^{-1}\al_3 z \ast y$
and
$z \ast x = \al_1\al_2\al_3^{-1} x \ast z$.
Hence we obtain 
\begin{equation}\label{etiso}
S^{\nu^{-1}} \cong k\langle x,y,z\rangle/(xy-\al_1^{-1}\al_2 \al_3 yx, yz-\al_1\al_2^{-1} \al_3 zy, zx-\al_1\al_2 \al_3^{-1} xz).
\end{equation}

Let
$R = k\langle x,y,z\rangle/(xy-\be_1 yx, yz-\be_2 zy,\ zx-\be_3 xz)$
be another standard graded skew polynomial algebra with $\be_1\be_2\be_3\neq 0$.
The following criteria are known to hold.
\begin{itemize}
\item $S \cong R$ if and only if 
\begin{equation}\label{espi}
(\be_1, \be_2, \be_3) \in
\left\{
\begin{array}{ccc}
(\al_1, \al_2, \al_3),&(\al_3, \al_1, \al_2),&(\al_2, \al_3, \al_1),\\[3pt]
(\al_1^{-1}, \al_3^{-1}, \al_2^{-1}),&(\al_3^{-1}, \al_2^{-1}, \al_1^{-1}),&(\al_2^{-1}, \al_1^{-1}, \al_3^{-1}) 
\end{array}
\right\}.
\end{equation}
\item $\GrMod S \simeq \GrMod R$ if and only if 
\begin{equation}\label{espm}
\be_1\be_2\be_3=(\al_1\al_2\al_3)^{\pm 1}.
\end{equation}
\end{itemize}
(For the first statement, see \cite[Theorem 2.4]{Ga} or \cite[Lemma 2.3]{Vi}, and for the second, see \cite[Theorem 4.7, Example 4.10]{Mo} or \cite[Theorem 2.5]{Vi}.)
Let $\mu$ be the Nakayama automorphism of $R$.
By the first criterion above and \eqref{etiso}, we see that $S^{\nu^{-1}} \cong R^{\mu^{-1}}$ if and only if 
\begin{equation}\label{etcl1}
(\be_1^{-1}\be_2\be_3,\,
 \be_1\be_2^{-1}\be_3,\,
 \be_1\be_2\be_3^{-1})
\in
\left\{
\begin{array}{ccc}
(\al_1^{-1}\al_2\al_3,\,
 \al_1\al_2^{-1}\al_3,\,
 \al_1\al_2\al_3^{-1}),\\[3pt]
(\al_1\al_2\al_3^{-1},\,
 \al_1^{-1}\al_2\al_3,\,
 \al_1\al_2^{-1}\al_3),\\[3pt]
(\al_1\al_2^{-1}\al_3,\,
 \al_1\al_2\al_3^{-1},\,
 \al_1^{-1}\al_2\al_3),\\[3pt]
(\al_1\al_2^{-1}\al_3^{-1},\,
 \al_1^{-1}\al_2^{-1}\al_3,\,
 \al_1^{-1}\al_2\al_3^{-1}),\\[3pt]
(\al_1^{-1}\al_2^{-1}\al_3,\,
 \al_1^{-1}\al_2\al_3^{-1},\,
 \al_1\al_2^{-1}\al_3^{-1}),\\[3pt]
(\al_1^{-1}\al_2\al_3^{-1},\,
 \al_1\al_2^{-1}\al_3^{-1},\,
 \al_1^{-1}\al_2^{-1}\al_3)
\end{array}
\right\}.
\end{equation}
It is straightforward to check that \eqref{etcl1} is equivalent to
\begin{equation}\label{etcl2}
(\be_1^2,\, \be_2^2,\, \be_3^2,\, \be_1\be_2\be_3)
\in
\left\{
\begin{array}{ccc}
(\al_1^2,\,\al_2^2,\,\al_3^2,\,\al_1\al_2\al_3),\\[3pt]
(\al_3^2,\,\al_1^2,\,\al_2^2,\,\al_1\al_2\al_3),\\[3pt]
(\al_2^2,\,\al_3^2,\,\al_1^2,\,\al_1\al_2\al_3),\\[3pt]
(\al_1^{-2},\,\al_3^{-2},\,\al_2^{-2},\,(\al_1\al_2\al_3)^{-1}),\\[3pt]
(\al_3^{-2},\,\al_2^{-2},\,\al_1^{-2},\,(\al_1\al_2\al_3)^{-1}),\\[3pt]
(\al_2^{-2},\,\al_1^{-2},\,\al_3^{-2},\,(\al_1\al_2\al_3)^{-1})
\end{array}
\right\}.
\end{equation}
From the above, we obtain the implications
\[ \eqref{espi} \Longrightarrow \eqref{etcl2} \Longrightarrow \eqref{espm},
\]
while the converses do not hold in general.
These illustrate the implications
\[(a) \Longrightarrow (c) \Longrightarrow (d)\]
in Corollary \ref{cmain2}.
In particular, it follows that the converses of
\[(a) \Longrightarrow (b)\quad \text{and} \quad (b) \Longrightarrow (d)\]
do not hold in general.
\end{ex}

We next consider examples from another perspective.
Let $S$ be a Koszul AS-regular algebra of dimension $d\ge 1$, $\sigma \in \GrAut(S)$, and $L := S_\sigma(-1) \in \GrMod S^{\en}$.
Then the Jacobson radical of $(S^{\sigma^{-1}})^!$ is nonzero, so $(S^{\sigma^{-1}})^!$ is not semisimple. 
Since $(S^{\sigma^{-1}})^!$ is a self-injective algebra, it follows that $\gldim (S^{\sigma^{-1}})^! = \infty$.
Therefore, $\uCM(S \ltimes L) \simeq \D^b(\mod ((S^{\sigma^{-1}})^!)^{\op})$ has neither a Serre functor nor Auslander-Reiten triangles.
On the other hand, the subcategory
$\thick(\Omega^d (S \ltimes L)_0(d)) \simeq \K^b(\proj ((S^{\sigma^{-1}})^!)^{\op})$
admits a Serre functor and Auslander-Reiten triangles by Theorem \ref{tmain1a}.

\begin{ex}\label{eAR1}
Consider the case $S = k[x]$ with $\deg x = 1$.
Then $S \ltimes \omega_S$ is given by
$A = k[x,y]/(y^2)$ with $\deg x = \deg y = 1$.
In this case, $\thick(\Omega^1 k(1)) \subset \uCM(A)$ coincides with $\uCMz(A)$,
the stable category of graded maximal Cohen-Macaulay modules that are locally free at all non-maximal prime ideals of $A$ by \cite[Proposition 2.5(d)]{BIY}.

For $i \ge 1$, let $I_i = \langle x^i, y \rangle$ be a graded ideal of $A$.
For $i \ge 1$ and $j \in \ZZ$, set
$X_{i,j} := \Omega^j I_i(i) \cong I_i(i-j)$.
Then \cite[Proposition 2.5(f)]{BIY} shows that the Auslander-Reiten quiver of $\thick(\Omega^1 k(1))$ is as follows:
\begin{align}\label{e.quiv}
\begin{split}
\xymatrix@R=0cm@C=0.275cm{
\cdots  & X_{1,2} \ar[dr] && X_{1,1} \ar[dr] && X_{1,0} \ar[dr] && X_{1,-1} \ar[dr] && X_{1,-2} \ar[dr] & \cdots \\
&\cdots & X_{2,2} \ar[dr] \ar[ur] && X_{2,1} \ar[dr] \ar[ur] && X_{2,0} \ar[dr] \ar[ur] && X_{2,-1} \ar[dr] \ar[ur] && X_{2,-2} & \cdots \\
\cdots & X_{3,3} \ar[dr] \ar[ur] && X_{3,2} \ar[dr] \ar[ur]&& X_{3,1} \ar[dr] \ar[ur] && X_{3,0} \ar[dr] \ar[ur] && X_{3,-1} \ar[dr] \ar[ur]& \cdots \\
&\cdots & X_{4,3} \ar[ur] && X_{4,2} \ar[ur] && X_{4,1}\ar[ur] && X_{4,0} \ar[ur] && X_{4,-1} &\cdots \\
&& \vdots && \vdots && \vdots && \vdots && \vdots
}
\end{split}
\end{align}
\end{ex}

\begin{ex}\label{eAR2}
We consider a generalization of Example \ref{eAR1}.
Let $R=k[x]$ with $\deg x=1$.
Define a graded algebra automorphism $\varphi \in \GrAut(R)$ by $\varphi(x) = \zeta x$, where $\zeta$ is a primitive $n$-th root of unity.
Let $G=\ang{\vph} \leq \GrAut (R)$.
Consider the graded skew group algebra $S=R *G$.
By \cite[Proposition 2.27]{MUs}, $S$ is a Koszul AS-regular algebra of dimension $1$.
We see that $S^! \cong R^! * G \cong k[z]/(z^2)*G\cong kQ/(z^2)$, where $Q$ is the quiver
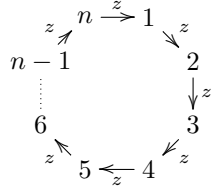
\[\small
\begin{xy}
0;<2.4pt,0pt>:<0pt,2.4pt>:: 
(0,0) *+{n}="a",
(10,0) *+{1}="b",
(17,-7) *+{2}="c",
(17,-17) *+{3}="d",
(10,-24) *+{4}="e",
(0,-24) *+{5}="f",
(-7,-17) *+{6}="g",
(-7,-7) *+{n-1}="h",
\ar "a";"b"^{z}
\ar "b";"c"^{z}
\ar "c";"d"^{z}
\ar "d";"e"^{z}
\ar "e";"f"^{z}
\ar "f";"g"^{z}
\ar@{.} "g";"h"
\ar "h";"a"^{z}
\end{xy}
\]
Let $A =S \ltimes S(-1) \cong S[y]/(y^2)$ with $\deg y = 1$. Then $\thick(\Omega^1 A_0(1)) \subset \uCM(A)$ is triangle equivalent to $\K^b(\proj kQ/(z^2))$ by Theorem \ref{tmain1a}.
By \cite[Proposition 2.5(c)]{BIY}, we see that the Auslander-Reiten quiver of $\thick(\Omega^1 A_0(1))$ consists of $n$ components, each of which has the same shape as \eqref{e.quiv}.
\end{ex}

\begin{ex}\label{eAR3}
Let $S = k\ang{x,y}/(xy+yx)$ be the standard graded $(-1)$-skew polynomial algebra in two variables.
By Example \ref{exntd2}(1), we have $S \cong S^{\nu^{-1}}$ and $S^! \cong k[x,y]/(x^2,y^2)$.
By Corollary \ref{cmain1}, $\thick(\Omega^2 k(2)) \subset \uCM(S \ltimes \om_S(1))$ is equivalent to $\K^b(\proj k[x,y]/(x^2,y^2))$.
Since $\K^b(\proj k[x,y]/(x^2,y^2))$ has a Serre functor isomorphic to the identity, $\thick(\Omega^2 k(2))$ is a $0$-Calabi-Yau triangulated category.
In contrast to Example \ref{eAR1}, the Auslander-Reiten quiver of $\thick(\Omega^2 k(2))$ appears to be considerably more complicated.
\end{ex}

\section*{Acknowledgments}
The author is grateful to Kota Yamaura and Issei Enomoto for valuable discussions and comments.
This work was supported by JSPS KAKENHI Grant Number 26K06761.


\begin{thebibliography}{99}
\bibitem{AIR}
C. Amiot, O. Iyama, and I. Reiten,
\textit{Stable categories of Cohen-Macaulay modules and cluster categories},
Amer. J. Math. \textbf{137} (2015), no. 3, 813--857.

\bibitem{AW}
D. D. Anderson and M. Winders,
\textit{Idealization of a module},
J. Commut. Algebra \textbf{1} (2009), no. 1, 3--5.

\bibitem{AZ}
M. Artin and J. J. Zhang,
\textit{Noncommutative projective schemes},
Adv. Math. \textbf{109} (1994), no. 2, 228--287.

\bibitem{BGS}
A. Beilinson, V. Ginzburg, and W. Soergel,
\textit{Koszul duality patterns in representation theory},
J. Amer. Math. Soc. \textbf{9} (1996), no. 2, 473--527.

\bibitem{Bu}
R.-O. Buchweitz,
\textit{Maximal Cohen-Macaulay modules and Tate cohomology, With appendices and an introduction by Luchezar L. Avramov, Benjamin Briggs, Srikanth B. Iyengar and Janina C. Letz}, 
Math. Surveys Monogr., 262, American Mathematical Society, Providence, RI, 2021.

\bibitem{BEH}
R.-O. Buchweitz, D. Eisenbud, and J. Herzog, 
\textit{Cohen-Macaulay modules on quadrics},
Singularities, representation of algebras, and vector bundles (Lambrecht, 1985), 58--116. Lecture Notes in Math., 1273
Springer-Verlag, Berlin, 1987.

\bibitem{BIY}
R.-O. Buchweitz, O. Iyama, and  K. Yamaura,
\textit{Tilting theory for Gorenstein rings in dimension one},
Forum Math. Sigma \textbf{8} (2020), Paper No. e36, 37 pp.

\bibitem{Ch}
X.-W. Chen, 
\textit{Graded self-injective algebras ``are'' trivial extensions},
J. Algebra \textbf{322} (2009), no. 7, 2601--2606.

\bibitem{FGR}
R. M. Fossum, P. A. Griffith, and I. Reiten,
\textit{Trivial extensions of abelian categories:
Homological algebra of trivial extensions of abelian categories with applications to ring theory},
Lecture Notes in Math., Vol. 456, Springer-Verlag, Berlin-New York, 1975.

\bibitem{Ga}
J. Gaddis,
\textit{The isomorphism problem for quantum affine spaces, homogenized quantized Weyl algebras, and quantum matrix algebras},
J. Pure Appl. Algebra \textbf{221} (2017), no. 10, 2511--2524.

\bibitem{Han}
N. Hanihara,
Auslander correspondence for triangulated categories,
Algebra Number Theory \textbf{14} (2020), no. 8, 2037--2058.

\bibitem{Hapb}
D. Happel,
\textit{Triangulated categories in the representation theory of finite-dimensional algebras},
London Mathematical Society Lecture Note Series, vol. 119, Cambridge University Press, Cambridge, 1988.

\bibitem{Hap2}
D. Happel,
\textit{On Gorenstein algebras},
Representation theory of finite groups and finite-dimensional algebras (Bielefeld, 1991), 389--404, Progr. Math., 95
Birkh\"auser Verlag, Basel, 1991.

\bibitem{HIMO}
M. Herschend, O. Iyama, H. Minamoto, and S. Oppermann,
\textit{Representation theory of Geigle-Lenzing complete intersections},
Mem. Amer. Math. Soc. \textbf{285} (2023), no. 1412, vii+141 pp.

\bibitem{IKU}
O. Iyama, Y. Kimura, and K. Ueyama,
\textit{Cohen-Macaulay representations of Artin-Schelter Gorenstein algebras of dimension one},
preprint, \texttt{arXiv:2404.05925v4}.

\bibitem{IT}
O. Iyama and R. Takahashi,
\textit{Tilting and cluster tilting for quotient singularities},
Math. Ann. \textbf{356} (2013), no. 3, 1065--1105.

\bibitem{JoL}
P. J\o rgensen,
\textit{Local cohomology for non-commutative graded algebras},
Comm. Algebra \textbf{25} (1997), no. 2, 575--591.

\bibitem{JoC}
P. J\o rgensen,
\textit{Properties of AS-Cohen-Macaulay algebras},
J. Pure Appl. Algebra \textbf{138} (1999), no. 3, 239--249.

\bibitem{JoLF}
P. J\o rgensen,
\textit{Linear free resolutions over non-commutative algebras},
Compos. Math. \textbf{140} (2004), no. 4, 1053--1058.

\bibitem{JoB}
P. J\o rgensen, 
\textit{A noncommutative BGG correspondence},
Pacific J. Math. \textbf{218} (2005), no. 2, 357--377.

\bibitem{Le}
T. Levasseur,
\textit{Some properties of noncommutative regular graded rings},
Glasgow Math. J. \textbf{34} (1992), no. 3, 277--300.

\bibitem{MS}
R. Mart\'inez Villa and M. Saor\'in,
\textit{Koszul equivalences and dualities},
Pacific J. Math. \textbf{214} (2004), no. 2, 359--378.

\bibitem{MR}
J. C. McConnell and J. C. Robson,
\textit{Noncommutative Noetherian rings},
With the cooperation of L. W. Small, Revised edition.
Grad. Stud. Math., 30, American Mathematical Society, Providence, RI, 2001.

\bibitem{MY}
H. Minamoto and K. Yamaura,
\textit{On finitely graded Iwanaga-Gorenstein algebras and the stable categories of their (graded) Cohen-Macaulay modules}
Adv. Math. \textbf{373} (2020), 107228, 57 pp.

\bibitem{MoP}
I. Mori,
\textit{Rationality of the Poincar\'e series for Koszul algebras},
J. Algebra \textbf{276} (2004), no. 2, 602--624.

\bibitem{MoR}
I. Mori,
\textit{Riemann-Roch like theorem for triangulated categories}
J. Pure Appl. Algebra \textbf{193} (2004), no. 1-3, 263--285.

\bibitem{Mo}
I. Mori,
\textit{Noncommutative projective schemes and point schemes},
Algebras, Rings and Their Representations, 215--239,
World Scientific Publishing Co. Pte. Ltd., Hackensack, NJ, 2006.

\bibitem{MUs}
I. Mori and K. Ueyama, 
\textit{Stable categories of graded maximal Cohen-Macaulay modules
over noncommutative quotient singularities},
Adv. Math. \textbf{297} (2016), 54--92.

\bibitem{MUk} 
I. Mori and K. Ueyama, 
\textit{Noncommutative Kn\"orrer's periodicity theorem and noncommutative quadric hypersurfaces},
Algebra Number Theory \textbf{16} (2022), no. 2, 467--504.

\bibitem{RV}
I. Reiten and M. Van den Bergh,
\textit{Noetherian hereditary abelian categories satisfying Serre duality}
J. Amer. Math. Soc. \textbf{15} (2002), no. 2, 295--366.

\bibitem{RR}
M. L. Reyes and D. Rogalski,
\textit{Graded twisted Calabi-Yau algebras are generalized Artin-Schelter regular}
Nagoya Math. J. \textbf{245} (2022), 100--153.

\bibitem{RZ}
R. Rouquier and A. Zimmermann,
\textit{Picard groups for derived module categories},
Proc. London Math. Soc. (3) \textbf{87} (2003), no. 1, 197--225.

\bibitem{ST}
B. Shelton and C. Tingey,
\textit{On Koszul algebras and a new construction of Artin-Schelter regular algebras},
J. Algebra \textbf{241} (2001), no. 2, 789--798.

\bibitem{SY}
A. Skowro\'nski and K. Yamagata,
\textit{Frobenius algebras II: Tilted and Hochschild extension algebras},
EMS Textbk. Math., European Mathematical Society (EMS), Z\"urich, 2017. 

\bibitem{Sm}
S. P. Smith,
\textit{Some finite-dimensional algebras related to elliptic curves},
Representation theory of algebras and related topics (Mexico City, 1994), 315--348, 
CMS Conf. Proc. 19, Amer. Math. Soc., Providence, RI, 1996.

\bibitem{SV}
S. P. Smith and M. Van den Bergh,
\textit{Noncommutative quadric surfaces},
J. Noncommut. Geom. \textbf{7} (2013), no. 3, 817--856.

\bibitem{Va}
M. Van den Bergh, 
\textit{Existence theorems for dualizing complexes over non-commutative graded and filtered rings},
J. Algebra \textbf{195} (1997), no. 2, 662--679.

\bibitem{Vi}
J. Vitoria,
\textit{Equivalences for noncommutative projective spaces},
preprint, \texttt{arXiv:1001.4400v3}.

\bibitem{Ye}
A. Yekutieli,
\textit{Dualizing complexes over noncommutative graded algebras},
J. Algebra \textbf{153} (1992), no. 1, 41--84.

\bibitem{Zh}
J. J. Zhang,
\textit{Twisted graded algebras and equivalences of graded categories},
Proc. London Math. Soc. (3)  \textbf{72} (1996), no. 2, 281--311.
\end{thebibliography}
\end{document}